\documentclass[12pt]{article}
\usepackage{graphicx,amssymb,xcolor}
\usepackage[bookmarks=true]{hyperref}
\newcommand{\gendis}{\kappa}

\newcommand{\beq}{\begin{eqnarray}}
\newcommand{\eeq}{\end{eqnarray}}
\newcommand{\dsp}{\displaystyle}
\newcommand{\bd}{\begin{displaymath}}
\newcommand{\be}{\begin{equation}}
\newcommand{\ba}{\begin{array}}
\newcommand{\ed}{\end{displaymath}}
\newcommand{\ee}{\end{equation}}
\newcommand{\ea}{\end{array}}
\newcommand{\espace}{\mbox{ }}

\newcommand{\Prob}{{\rm I\hspace{-0.8mm}P}}
\newcommand{\Exp}{{\rm I\hspace{-0.8mm}E}}
\newcommand{\indicator}[1]{{\mbox{\large\bf$1$}}_{#1}}

\newcommand{\sgn}{\mbox{\rm sgn}\,}

\def\N{\mathbb{N}}
\def\Z{\mathbb{Z}}
\def\R{\mathbb{R}}

\newcommand{\eqref}[1]{(\ref{#1})}

\newtheorem{theorem}{Theorem}[section]

\newtheorem{assumption}{Assumption}[section]

\newtheorem{proposition}{Proposition}[section]
\newtheorem{definition}{Definition}[section]
\newtheorem{lemma}{Lemma}[section]
\newtheorem{corollary}{Corollary}[section]

\newtheorem{remark}{Remark}[section]
\newtheorem{example}{Example}[section]

\newenvironment{proof}[2]{\espace\\{\em Proof of #1 \ref{#2}.}}{\hfill\mbox{$\square$}}
\begin{document}
\title{Hydrodynamics in a condensation regime: the disordered asymmetric zero-range process}
\author{C. Bahadoran$^{a,e}$, T. Mountford$^{b,e}$, K. Ravishankar$^{c,e}$, E. Saada$^{d,e}$}
\date{}
\maketitle
$$ \ba{l}
^a\,\mbox{\small Laboratoire de Math\'ematiques Blaise Pascal, Universit\'e Clermont Auvergne,} \\
\quad \mbox{\small 63177 Aubi\`ere, France. E-mail:
Christophe.Bahadoran@uca.fr}\\
^b\, \mbox{\small Institut de Math\'ematiques, \'Ecole Polytechnique F\'ed\'erale, 
} \\
\quad \mbox{\small Lausanne, Switzerland. E-mail:
thomas.mountford@epfl.ch}\\
^c\, \mbox{\small NYU-ECNU
Institute of Mathematical Sciences at NYU Shanghai,}\\
\quad\mbox{\small 3663 Zhongshan Road North, Shanghai, 200062, China. E-mail:
 kr26@nyu.edu}\\
%
^d\, \mbox{\small 
CNRS, UMR 8145, MAP5, Universit\'e Paris Descartes, 45 rue des Saints-P\`eres,}\\
\quad\mbox{\small 
75270 Paris cedex 06, France. E-mail:
Ellen.Saada@mi.parisdescartes.fr}\\
^e\, \mbox{\small Centre Emile Borel, Institut Henri Poincar\'e, 75005 Paris, France.} \\ \\
\ea
$$
\begin{abstract}
We study asymmetric zero-range processes on $\Z$ with nearest-neighbour 
jumps and site disorder. The jump rate of particles is an arbitrary but bounded 
nondecreasing function of the number of particles. For any given environment satisfying
suitable averaging properties, we establish a 
hydrodynamic limit given by a scalar conservation law {\em including} 
the domain above critical density, where the flux is shown to be constant. 
\end{abstract}
\section{Introduction}
The asymmetric zero-range process (AZRP) with site disorder was introduced in \cite{ev}
 (in connection with condensation phenomena),  
and has since then attracted strong interest due to its phase transition first described
 in \cite{fk}. 
This phase transition is one of the reasons why the hydrodynamic behaviour of this process is still
a partially open problem. 
 This  question is addressed in this paper, 
and we believe this is the first rigorous result in this direction embedded in a general framework.\\ \\ 
The  AZRP with site disorder
is defined by a nondecreasing jump rate  function $g:\N\to\N$, a function 
 $\alpha :\Z^d\to\R_+$ (called the environment or disorder),
and a jump distribution $p(.)$ on $\Z^d$, for $d\ge 1$.   
A particle leaves site $x$ at rate $\alpha(x)g[\eta(x)]$, where $\eta(x)$ denotes the 
current number of particles at $x$, and moves to $x+z$, where $z$ is chosen at random 
with distribution $p(.)$. This  model has product invariant measures; 
it exhibits a critical density  $\rho_c$  if the function $g$ is bounded,
and if  $g$ and $\alpha$ satisfy some averaging properties
plus a proper tail assumption. 
 In the models we consider 
in this paper (see next section for greater precision), 
  $d=1$,  there will be  a unique equilibrium
for each density up to this critical value $\rho_c$  
and no equilibria of strictly higher density.
For instance in the case  $g(.) \equiv \ 1 $ and $p(.)$ concentrated on the value $1$,
we obtain $ M/M/1 $ queues in tandem, for which
 \cite{afgl} showed that there were no invariant measures of supercritical density. \\ \\
A signature of phase transition arises in the hydrodynamic limit. For asymmetric 
conservative systems with local interactions,
one usually expects (see e.g. \cite{rez,kl}) a hydrodynamic limit given by entropy solutions 
of a scalar conservation law
\be\label{burgers_intro}
\partial_t\rho(t,x)+\partial_x[f(\rho(t,x))]=0
\ee
where $\rho(t,x)$ is the local particle density field, and $f(\rho)$ is 
the flux-density relation determined by the microscopic dynamics. 
 For the site-disordered AZRP  in any space dimension, the hydrodynamic limit was studied 
in \cite{bfl} but only in the case 
 where phase transition does not occur  (that is $\rho_c=+\infty$).
It was shown to be still given by \eqref{burgers_intro}, with an effective flux function 
depending on the disorder distribution.
 The hydrodynamic limit including phase transition 
was studied by \cite{ks}  for  $M/M/1$ queues in tandem. 
It was shown there that one still has \eqref{burgers_intro}, but 
phase transition is indicated by a plateau on the flux function $f$ for $\rho\in[\rho_c,+\infty)$, 
where this function takes a constant value 
 $c>0$, which is the infimum of the support of the 
distribution of $\alpha(0)$, that is the slowest service rate, achieved only 
asymptotically by faraway servers in both directions (see
Subsection \ref{subsec:efflux} for  a precise statement).  Both \cite{bfl, ks} are quenched results established for {\em almost} every realization of a random i.i.d. site disorder.
A similar  flat region  was predicted 
(albeit not established to this day) in \cite{hs} 
for non-monotone spatially homogeneous zero-range processes.
Both  \cite{hs} and our model 
correspond to condensation regimes, though of different natures 
(localized in our case but uniformly distributed in the case of \cite{hs}).
As pointed out in \cite{hs}, the hydrodynamic limit in such a regime falls 
outside the scope of standard local-equilibrium based approaches.  
On the mathematical side,  related  references on condensation 
in asymmetric zero-range processes include \cite{cg,lan3}. \\ \\
In this paper,
we 
extend the result of  \cite{ks}  to a large class of  site-disordered
AZRP, namely,
with nearest-neighbour  
jumps, not necessarily totally asymmetric, and general 
jump rate function $g$. 
Moreover, we go beyond the case of an ergodic disorder by giving
optimal conditions on a {\it given} environment for the 
hydrodynamic limit, and show that the location of the transition can be influenced
by zero-density defects, invisible on the limiting empirical distribution of the environment. 
To achieve our results,   we show that the missing \footnote{supercritical}
equilibria can be replaced by weaker {\em pseudo-equilibria}, 
and we introduce an {\em interface process}
that gives a new point of view of the microscopic density profile.
 We point out that the scaling limit of the interface process, 
 which comes in parallel to the hydrodynamic limit, contains more information than the latter, 
in particular, the motion of microscopic characteristics.  
However, we leave a precise
description of this to a future paper, where it will be investigated in full generality. 
 \\ \\
 Also partly conveyed by the interface process is the {\em local equilibrium} property, 
that is the natural question following the derivation of the hydrodynamic limit. 
This property is studied in depth in the companion paper \cite{bmrs4}.
Note that the situation is more delicate than usual in that the ``freezing''
of supercritical areas in the hydrodynamic scaling does not have local implications.
In fact locally we see (in various forms) the convergence to the 
upper equilibrium measure,  which has lower density.  
 \\ \\
The paper is organized as follows. In Section \ref{sec_results}, we introduce the model 
and notation, and state our  hydrodynamic result. We comment and illustrate the latter
in Section \ref{sec:Gunter}. 
In Section \ref{sec_proof_hydro}, we prove it.  Finally, some technical results are proved in Appendices \ref{app:justify} and \ref{app:uniform}.  
\section{Notation and results}
\label{sec_results}
In the sequel,  $\R$ denotes the set of real numbers, 
$\Z$  the set of signed integers, $\N=\{0,1,\ldots\}$ 
the set of nonnegative integers, 
 and $\overline{\N}:=\N\cup\{+\infty\}$. 
For $x\in\R$,  $\lfloor x\rfloor$  
denotes the integer part of $x$, that is largest integer $n\in\Z$ 
such that $n\leq x$, and $\delta_x$ denotes the Dirac measure at $x$. 
If $f$ is a real-valued function defined on an interval $I$ of $\R$, and $x\in I$, we denote by
$$f(x+):=\lim_{y\to x,\,y>x}f(y),\quad\mbox{resp.}\quad
f(x-):=\lim_{y\to x,\,y<x}f(y)$$
the right (resp. left) limit of $f$ at $x$,  whenever this makes sense given the position of $x$ in $I$. 
The notation $X\sim\mu$ means that a random 
variable $X$ has probability distribution $\mu$.\\ \\
Let   $\overline{\mathbf{X}}:=\overline{\N}^{\Z}$   denote the set of particle configurations, 
and ${\mathbf{X}}:=\N^\Z$ the subset of particle configurations with finitely many particles at each site.
 A configuration in $\overline{\mathbf{X}}$ is of the form  $\eta=(\eta(x):\,x\in\Z)$ 
 where $\eta(x)\in\overline{\N}$ for each $x\in\Z$. 
 The set  $\overline{\mathbf{X}}$  is equipped with the  coordinatewise  order:  for  
 $\eta,\xi\in\overline{\mathbf{X}}$,  we write 
$\eta\leq\xi$ if and only if $\eta(x)\leq\xi(x)$ for every $x\in\Z$;  in the latter 
inequality, $\leq$ stands for extension to $\overline{\N}$ of
the natural order on $\N$, defined by $n\leq+\infty$ for every $n\in\N$, and $+\infty\leq+\infty$. 
  This order is extended to probability measures on $\overline{\mathbf{X}}$: For two 
 probability measures $\mu,\nu$, we write $\mu\le\nu$ if and only if
$\int fd\mu\le\int fd\nu$ for any nondecreasing function $f$ on $\overline{\mathbf{X}}$.  
\subsection{The process and its invariant measures}\label{subsec:procinv}
Let  $p(.)$ be a probability measure on $\Z$ supported on $\{-1,1\}$.
 We set $p:=p(1)$, $q=p(-1)=1-p$, and assume $p\in(1/2,1]$,
so that the mean drift of the associated random walk is $p-q>0$.\\ \\
Let $g:\N\to[0,+\infty)$ be a nondecreasing function such that
\be\label{properties_g}
g(0)=0<g(1)\leq \lim_{n\to+\infty}g(n)=:g_\infty<+\infty
\ee
We extend $g$ to  $\overline{\N}$  by setting $g(+\infty)=g_\infty$.
Without loss of generality, we henceforth assume 
$g(+\infty)=g_\infty=1$. \\ \\
Let 
$\alpha=(\alpha(x),\,x\in\Z)$ (called the environment or disorder) be a 
 $[0,1]$-valued sequence. 
The set of environments is denoted by
\be\label{set_env}
{\bf A}:=[0,1]^{\Z}
\ee
We consider the 
Markov process  $(\eta_t^\alpha)_{t\geq 0}$
on  $\overline{\mathbf{X}}$  with generator given for any cylinder function  
$f:\overline{\mathbf{X}}\to\R$  by
\be\label{generator}
L^\alpha f(\eta)  =  \sum_{x,y\in\Z}\alpha(x)
p(y-x)g(\eta(x))\left[
f\left(\eta^{x,y}\right)-f(\eta)
\right]\ee
where, if  $\eta(x)>0$, $\eta^{x,y}$ 
denotes the new configuration obtained from $\eta$ after a particle has 
jumped from $x$ to $y$.
%
 In cases of infinite particle number,  the following interpretations hold:
if $\eta(x)<\eta(y)=+\infty$, $\eta^{x,y}$ denotes the new configuration obtained from $\eta$ after a particle has 
been removed from $x$;
 if $\eta(x)=+\infty>\eta(y)$, $\eta^{x,y}$ denotes the new configuration obtained from $\eta$ after a particle has 
been added at $y$.\\ \\
%
This process has the property 
that if $\eta_0\in{\mathbf{X}}$, then almost surely, one has $\eta_t\in{\mathbf{X}}$ 
for every $t>0$. In this case, it may be considered as a Markov process on $\mathbf{X}$ with generator
\eqref{generator} restricted to functions $f:{\mathbf{X}}\to\R$. \\
 When the environment $\alpha(.)$ 
is identically equal to $1$, we recover
the {\em homogeneous} zero-range process (see \cite{and} for its detailed analysis). \\ \\ 
For the existence and uniqueness of  $(\eta_t^\alpha)_{t\geq 0}$  
see \cite[Appendix B]{bmrs2}.   
 Recall from \cite{and} that, since $g$ is nondecreasing, 
 $(\eta_t^\alpha)_{t\geq 0}$  
is attractive, i.e.
its semigroup maps nondecreasing functions (with respect to 
the partial order on $\overline{\mathbf{X}}$) onto nondecreasing functions.
One way to see this is to construct a monotone coupling of  two copies of the process, 
see Subsection \ref{subsec:Harris} below. \\ \\
We set
$$
g(n)!:=\prod_{k=1}^{n} g(k)
$$
for $n\in\N\setminus\{0\}$,
and $g(0)!:=1$. 
For ${\beta}<1$,  
we define the probability measure $\theta_{{\beta}}$ on $\N$
 by \label{properties_a}
\begin{equation}\label{eq:theta-lambda}
\theta_{{\beta}}(n):=Z({\beta})^{-1}\frac{\beta^n}{g(n)!},\quad n\in\N,
\qquad\mbox{where}\quad Z(\beta):=\sum_{\ell=0}^{+\infty}\frac{\beta^\ell}{g(\ell)!}
\end{equation}
We denote by  $\mu_\beta^\alpha$  the invariant measure 
of $L^\alpha$ defined (see e.g. \cite{bfl}) as the product measure 
with marginal $\theta_{\beta/\alpha(x)}$ at site $x$: 
\be\label{def_mu_lambda_alpha}
\mu^\alpha_{\beta}(d\eta):=\bigotimes_{x\in\Z}\theta_{\beta/\alpha(x)}[d\eta(x)]
\ee
 Let
\be\label{inf_alpha}
c:=\inf_{x\in\Z}\alpha(x)
%
\ee
 The measure \eqref{def_mu_lambda_alpha} can be defined on $\overline{\mathbf{X}}$  for
\be\label{cond_simple_bar}
{\beta}\in[0,c]
\ee
by using the conventions
\be\label{extension_theta}\theta_1:=\delta_{+\infty}
\ee
\be
\label{convention_2}
\frac{\beta}{a}=0\mbox{ if }\beta=0\mbox{ and }a\geq 0
\ee
The measure \eqref{def_mu_lambda_alpha} is always supported on $\bf X$ if
\be\label{cond_lambda_simple}
\beta\in(0,c)\cup\{0\}
\ee
When  $\beta=c>0$, conventions  \eqref{extension_theta}--\eqref{convention_2} 
yield  a measure supported on configurations with infinitely many particles 
at all sites $x\in\Z$ that achieve the infimum in \eqref{inf_alpha}, and finitely many particles at other sites.
In particular, this measure is supported on $\bf X$  when the infimum in \eqref{inf_alpha} is not achieved. When $c=0$, the measure \eqref{def_mu_lambda_alpha} is supported on the empty configuration.
Since 
$(\theta_{{\beta}})_{{\beta}\in [0,1)}$ is an exponential family, we have that,  for $\beta\in[0,c]$, 
 \be\label{stoch_inc}
 \mu^\alpha_{{\beta}}\,
 \mbox{is weakly continuous and stochastically increasing in } \beta
 \ee
and that the mean value  of $\theta_{{\beta}}$, 
given  for $\beta\in[0,1)$  by 
\be\label{mean_density}
R({\beta}):=\sum_{n=0}^{+\infty}n\theta_{{\beta}}(n)
\ee
is  an analytic function on $[0,1)$, increasing from $0$ to $+\infty$,  extended (cf. \eqref{extension_theta}) by setting
$R(1)=+\infty$.  
The mean particle density at $x$ under $\mu_{{\beta}}^\alpha$ is defined  for $\beta\in[0,c]$  by
\be\label{mean_density_quenched}
R^\alpha(x,{\beta}):=\Exp_{\mu^\alpha_{\beta}}[\eta(x)]=R\left[\frac{{\beta}}{\alpha(x)}\right]
\ee
\subsection{ The effective flux}
\label{subsec:efflux}
From now on, we will assume that $\alpha$ satisfies the following  assumption. 
\begin{assumption}\label{assumption_ergo}
There exists a probability measure $Q_0$
on $\bf A$ such that
\be\label{eq:assumption_ergo}
Q_0=\lim_{n\to+\infty}\frac{1}{n+1}\sum_{x=-n}^0 \delta_{\alpha(x)}
=\lim_{n\to+\infty}\frac{1}{n+1}\sum_{x=0}^n \delta_{\alpha(x)}
\ee
\end{assumption}
If follows from Assumption \ref{assumption_ergo} that 
\be\label{support_Q}
C:=\inf{\rm supp}\,Q_0\geq\inf_{x\in\Z}\alpha(x)=c
\ee
Assumption \ref{assumption_ergo} is satisfied for instance in the case of an ergodic random environment:
\begin{example}\label{example_random}
Let $Q$ be a spatially ergodic probability measure on $\bf A$ with marginal $Q_0$
(for instance, $Q=Q_0^{\otimes\Z}$). Then, $Q$-almost every $\alpha\in\bf A$ satisfies
Assumption \ref{assumption_ergo} and equality in \eqref{support_Q}. 
\end{example}
Example \ref{example_random} is special because it yields equality in \eqref{support_Q}. In 
Subsection \ref{subsec:flux} we give examples of deterministic environments satisfying 
Assumption \ref{assumption_ergo} for which \eqref{support_Q} is a strict inequality.\\ \\ 
For $\beta\in (0,c)\cup\{0\}$, using conventions 
 \eqref{extension_theta}--\eqref{convention_2}, we can define the
following quantity, which can be interpreted (see Lemma \ref{lemma_inter} below) as the average mean density 
under $\mu^\alpha_\beta$:
\be
\overline{R}^{\alpha}({\beta}):=
\lim_{n\to+\infty}\frac{1}{n+1}\sum_{x=-n}^0 R\left[\frac{{\beta}}{\alpha(x)}\right]
=\lim_{n\to+\infty}\frac{1}{n+1}\sum_{x=0}^n R\left[\frac{{\beta}}{\alpha(x)}\right]\label{average_afgl}
\ee
Indeed,
applying \eqref{eq:assumption_ergo} to the bounded continuous function 
$a\mapsto R[\beta/a]$, we obtain existence and equality of the above limits,
 and the equality 
\be\label{averages}
\overline{R}^\alpha(\beta)=\overline{R}^{Q_0}(\beta)
,\quad\forall \beta\in[0,c)
\ee
where 
\be
\overline{R}^{Q_0}({\beta}) 
:=  \int_{[0,1]}R\left[\frac{\beta}{a}\right]dQ_0[a]
=\int_{[C,1]}R\left[\frac{\beta}{a}\right]dQ_0[a]\in[0,+\infty],
\quad\forall  \beta\in[0,C]
\label{average_afgl_ergodic}
\ee
is also defined using conventions  \eqref{extension_theta}--\eqref{convention_2}. 
The function defined by \eqref{average_afgl_ergodic} is finite 
for $\beta\in(0,C)\cup\{0\}$, because the integrand in 
\eqref{average_afgl_ergodic} is bounded. If $0<\beta={C}$, 
the integral in \eqref{average_afgl_ergodic} may diverge at $C$. 
 Consequently, $\overline{R}^\alpha$ is finite for $\beta\in(0,c)\cup\{0\}$
and if $\beta=c<C$, but may be infinite if $\beta=c=C$. 
The function $\overline{R}^{Q_0}$ is increasing and continuous on the interval $[0,C]$
(see Lemma \ref{lemma_properties_flux} below), and so is 
$\overline{R}^\alpha$ on $(0,c)\cup\{0\}$ by \eqref{averages}. 
We may thus define inverses of $\overline{R}^{Q_0}$
and $\overline{R}^\alpha$ on their respective images.\\ \\
We define the critical density by 
\be\label{def_critical}
\rho_c^\alpha:=\sup\{\overline{R}^\alpha(\beta),\,\beta\in(0,c)\cup\{0\}\}\in[0,+\infty]
\ee
Thus $\rho_c^\alpha=0$ if $c=0$, whereas if $c>0$, we also have
\be\label{other_def_critical}
\rho_c^\alpha:=\overline{R}^\alpha(c-)
\ee
By \eqref{averages} and monotone convergence in \eqref{average_afgl_ergodic}, we  have
\be\label{also_critical}
\rho_c^\alpha=\int_{[0,1]}{R}\left[\frac{c}{a}\right]dQ_0(a)=
\int_{[C,1]}{R}\left[\frac{c}{a}\right]dQ_0(a)=
\int_{[c,1]}{R}\left[\frac{c}{a}\right]dQ_0(a)
\ee
where the  last  equality follows from \eqref{support_Q}.
\begin{remark}\label{remark_point_out_0}
While $\rho_c^\alpha$ is equal to the value obtained by letting $\beta=c$ 
in \eqref{average_afgl_ergodic}, it may not be obtained by letting $\beta=c$ 
in \eqref{average_afgl}. Indeed,
 the latter procedure may produce two different limits in \eqref{average_afgl},
  or a common limit different from the quantity defined by \eqref{def_critical} 
  (see example in Remark \ref{remark_point_out_00}).
\end{remark}
 For the reason explained in Remark \ref{remark_point_out_0}, it is relevant for us to 
define the value $\overline{R}^\alpha(c)$ when $c>0$ by the continuity extension
\be\label{bysetting}
\overline{R}^\alpha(c):=\overline{R}^\alpha(c-)=\rho_c^\alpha\in[0,+\infty]
\ee
and {\em not} by extending definition \eqref{average_afgl} to $\beta=c$. 
With \eqref{bysetting}, $\overline{R}^\alpha$ becomes an increasing continuous function from
$[0,c]$ to $[0,\rho_c^\alpha]$, and we may define its inverse from $[0,\rho_c]$ to $[0,c]$.
Then, for  $\beta\in(0,c)\cup\{0\}$,  
we may reindex the invariant measure $\mu^\alpha_{\beta}$ 
by the mean density   $\rho\in(0,\rho_c^\alpha)\cup\{0\}$,   by   
setting 
\be\label{invariant_density}
\mu^{\alpha,\rho}:=\mu^\alpha_{\left(\overline{R}^{\alpha}\right)^{-1}(\rho)}
\ee
 We now justify as announced the intepretation of \eqref{average_afgl} as the mean density.
\begin{lemma}\label{lemma_inter}
Let $\beta\in(0,c)\cup\{0\}$, and $\rho=\overline{R}^\alpha(\beta)\in(0,\rho_c^{\alpha})\cup\{0\}$.
Let $\eta^{\alpha,\rho}=\eta^\alpha_\beta$
be a random configuration in $\bf X$ with distribution $\mu^{\alpha,\rho}=\mu^\alpha_\beta$. 
Then the following limits hold in probability:
\be\label{justify_mean_beta}
\lim_{n\to+\infty}\frac{1}{n+1}\sum_{x=0}^n\eta^{\alpha}_\beta(x)=
\lim_{n\to+\infty}\frac{1}{n+1}\sum_{x=0}^n\eta^{\alpha}_\beta(-x)=\overline{R}^\alpha(\beta),
\ee
that is, 
\be\label{justify_mean}
\lim_{n\to+\infty}\frac{1}{n+1}\sum_{x=0}^n\eta^{\alpha,\rho}(x)=
\lim_{n\to+\infty}\frac{1}{n+1}\sum_{x=0}^n\eta^{\alpha,\rho}(-x)=\rho
\ee
\end{lemma}
\begin{remark}\label{remark_failure}
The limit \eqref{justify_mean_beta} may not hold in the case $\beta=c$. This is why we did 
not extend the reindexation \eqref{invariant_density} to this value of $\beta$ and 
$\rho=\overline{R}^\alpha(c)=\rho_c^\alpha$.
\end{remark}
Lemma \ref{lemma_inter} is proved in Appendix \ref{app:justify}.
We next define the subcritical part of the effective flux function as follows:
\begin{eqnarray}\label{def_flux}
f^\alpha(\rho) & = & (p-q)\left(\overline{R}^\alpha\right)^{-1}(\rho), \quad\forall\rho<\rho_c^{\alpha}
\end{eqnarray}
We extend the function $f^\alpha$ to densities $\rho\geq\rho_c$ by setting
\begin{eqnarray}\label{extension_flux}
f^\alpha(\rho)  & = &  (p-q)c, \quad \forall \rho\geq\rho_c^{\alpha}
\end{eqnarray}
 An alternative expression for \eqref{def_flux} is, for any $\rho<\rho_c^\alpha$ and $x\in\Z$, 
\begin{eqnarray}
\label{current_eq}
f^{\alpha}(\rho) &=&\int_{\mathbf{X}}\Big[p\alpha(x)g(\eta(x))-q\alpha(x+1)g[\eta(x+1)]\Big]
d\mu^{\alpha,\rho}(\eta)
\end{eqnarray}
which can be interpreted (cf. \eqref{standard_current}) as the mean current in the equilibrium state 
$\mu^{\alpha,\rho}$ with density $\rho$.
Indeed, it follows from \eqref{eq:theta-lambda} that 
\be\label{mean_rate}
\int_{\mathbf{X}}\alpha(x)g(\eta(x))d\mu^\alpha_{\beta}(\eta)
=\int_{\N}g(n)d\theta_{{\beta}}(n)={\beta}
\ee 
for all $x\in\Z$, $\alpha\in{\bf A}$ and $\beta\in  [0,c]$. 
 Then, by \eqref{mean_rate} and \eqref{invariant_density},
\be\label{as_will_be}
\quad\int_{\mathbf{X}}\Big[p\alpha(x)g(\eta(x))-q\alpha(x+1)g[\eta(x+1)]\Big]d\mu^\alpha_{\beta}(\eta)
=(p-q)\left(\overline{R}^{\alpha}\right)^{-1}(\rho)
\ee
 Some properties of the flux function are stated in Lemma \ref{lemma_properties_flux}.
 In the sequel, we shall often omit the superscript $\alpha$, and write  $\overline{R}$,  $f$ and $\rho_c$. 
\subsection{The hydrodynamic limit}
We first recall some standard 
definitions in hydrodynamic limit theory. We denote by
$\mathcal M(\R)$ the set of Radon measures on $\R$. To a particle configuration
 $\eta\in{\mathbf{X}}$, we associate a sequence of empirical measures 
 $(\pi^N(\eta):\,N\in\N\setminus\{0\})$ defined by
\[
\pi^N(\eta):=\frac{1}{N}\sum_{y\in\Z}\eta(y)\delta_{y/N}\in\mathcal M(\R)
\]
Let $\rho_0(.)\in L^\infty(\R)$, and let $(\eta^N_0)_{N\in\N\setminus\{0\}}$ denote 
a sequence of $\mathbf{X}$-valued random variables. We say this sequence has 
limiting density profile $\rho_0(.)$, if the sequence of empirical measures 
$\pi^N(\eta^N_0)$ converges in probability to the deterministic measure 
$\rho_0(.)dx$ with respect to the topology of vague convergence. 
We can now state our  result. 
The following additional assumption on the environment will be required for one of the statements. 
\begin{assumption}\label{assumption_dense} We say that the environment $\alpha$ 
has  {\em macroscopically dense defects} if 
there exists a sequence  of sites  $(x_n)_{n\in\Z}$  such that 
\be\label{assumption_afgl}
\forall n\in\Z,\,x_n<x_{n+1};\quad
\lim_{n\to\pm\infty}\alpha(x_n)=c
\ee
and
\be\label{assumption_afgl_0}
\lim_{n\to\pm\infty}\frac{x_{n+1}}{x_n}=1
\ee
\end{assumption}
\begin{theorem}
\label{th_hydro}
Assume the environment $\alpha$ satisfies  Assumption \ref{assumption_ergo}, 
and the  sequence $(\eta^N_0)_{N\in\N\setminus\{0\}}$ has limiting density profile $\rho_0(.)\in L^\infty(\R)$. 
For each $N\in\N\setminus\{0\}$, let  $(\eta^{\alpha,N}_{t})_{t\geq 0}$  denote the process with initial 
configuration $\eta^N_0$ and generator \eqref{generator}. Assume either that 
the initial data is subcritical, that is $\rho_0(.)< \rho_c$; or, that the 
defect density Assumption \ref{assumption_dense} holds.
 Let $\rho(.,.)$ denote the entropy solution to 
\be\label{conservation_law}
\partial_t \rho(x,t)+\partial_x f[\rho(x,t)]=0
\ee
 with initial datum $\rho_0(.)$. 
Then for any $t>0$,
 the sequence  $(\eta^{\alpha,N}_{Nt})_{N\in\N\setminus\{0\}}$ 
 has limiting density profile $\rho(.,t)$.
\end{theorem}
\begin{remark}\label{remark_th_hydro}
 The existence of a sequence satisfying condition \eqref{assumption_afgl} is equivalent to
the property
\be\label{not_too_sparse}
\liminf_{x\to+\infty}\alpha(x)=\liminf_{x\to-\infty}\alpha(x)=c
\ee
for the constant $c$ in \eqref{inf_alpha}.
The additional requirement \eqref{assumption_afgl_0}  
 sets a restriction
on the sparsity of  slow sites (where by ``slow sites''  
we mean sites where the disorder variable becomes arbitrarily close 
or equal to the infimum value $c$).
The role of Assumption \ref{assumption_dense} will be discussed in Section \ref{sec:Gunter}.
In particular, we will see in Subsection \ref{subsec:defect} 
that this condition prevents macroscopic separation 
of slow sites, as the latter could result in a spatially heterogeneous 
conservation law in the hydrodynamic limit.
\end{remark}
{\em A special case.} Of special importance is the so-called 
{\em Riemann} problem, that is the case when the initial data $\rho_0(.)$ has the particular
form
\begin{equation}\label{eq:rie}
R_{\lambda,\rho}(x)=\lambda \mathbf 1_{\{x<0\}}+\rho \mathbf 1_{\{x\geq 0\}}
\end{equation}
for $\lambda,\rho\in\R$. In this case,
entropy solutions can be computed explicitely.
Namely, let $\lambda,\rho\in\R$,  and $h=(\widehat{f}')^{-1}$, 
 where $\widehat{f}$ denotes  the  convex (resp. concave)  envelope of $f$ 
 on $[\lambda,\rho]$ (resp. $[\rho,\lambda]$).  Then (see Proposition 
 \ref{proposition_1} below), the entropy solution is given by 
\be\label{entropy_riemann}
R_{\lambda,\rho}(x,t)=h\left(\frac{x}{t}\right)
\ee
As will be discussed in Section \ref{sec:Gunter}, the phase transition 
can be seen explicitly on such solutions in the form of a front of critical 
density moving to the right after blocking supercritical densities coming from the left.\\ \\
{\em Remarks on the proof of Theorem \ref{th_hydro}.} 
The difficulty of proving hydrodynamics comes from the absence of invariant measures 
and the condensation phenomenon at supercritical densities. This prevents us 
from using the traditional approach to hydrodynamic limits based on local equilibrium, 
because the latter property (\cite{bmrs4}) fails at supercritical densities. 
In that approach, a lattice approximation of the macroscopic profile is defined by 
{\em block averaging}. A lattice version of the macroscopic equation is then 
obtained using {\em block estimates}, thanks to which the microscopic flux function 
can be replaced by a function of the local block average. In our case, 
due to condensation, mesoscopic block densities can blow up around condensation 
sites and fail to reflect the hydrodynamic density.\\ \\
We shall circumvent the impossibility of using the usual approach thanks 
to the following new ideas. First, we shall show that for our purpose, 
we retain sufficient information by replacing the unavailable 
supercritical equilibria by ``pseudo-equilibria'', that are simply systems
 with supercritical homogeneous macroscopic density profile.
Next, in Subsection \ref{subsec:interface}, we define a lattice profile 
in a new way, replacing the usual {\em discrete} block average 
\[
\rho^{N,l}(x,t)=\frac{1}{2l+1}\sum_{y\in\Z:\,|y-\lfloor Nx \rfloor|\leq l}\eta^N_{Nt}(y)
\]
(by discrete we mean that possible values of $\rho^{N,l}$ are discretized densities)
by a {\em continuous} lattice density field $\rho^N(x)$ taking ``real'' density values,
 that is the interface process referred to in the introduction. 
In a space region where $\rho^N$ does not fluctuate much, the system is close to an equilibrium or pseudo-equilibrium configuration
with an {\em a priori} random density parameter (such a configuration is not necessarily a stationary state when the density is indeed random, see Remark \ref{remark_ab} below).\\ \\
%
This new point of view will be coupled to a reduction principle introduced in \cite{bgrs}
 (see also \cite{bgrs2}--\cite{bgrs5}), 
where we showed that proving hydrodynamic limit for the Cauchy problem boils down to 
proving it for the Riemann problem, which can be analyzed more directly. The passage from 
Riemann to Cauchy problem can then be carried out 
in a way similar in spirit to Riemann-based numerical schemes for scalar conservation laws, 
by controlling the propagation of the error committed at successive time steps, when, 
replacing the actual entropy solution with a suitable piecewise constant approximation.
\section{Discussion and examples}\label{sec:Gunter}
 In this section, we   shed more light on Theorem \ref{th_hydro} 
 by giving examples of environments, flux functions and entropy solutions, and
illustrating the role of Assumption  \ref{assumption_dense}.  
\subsection{The flux function}\label{subsec:flux}
 We start by stating basic properties of the flux function. 
 So far, we have defined the critical density $\rho_c^\alpha$ 
 (cf.  \eqref{def_critical}--\eqref{other_def_critical}) 
and the flux function $f^\alpha$ (cf. 
\eqref{def_flux}--\eqref{extension_flux}) associated 
with an environment $\alpha$ satisfying Assumption \ref{assumption_ergo}.
These can be embedded in  the following family of critical densities and flux functions parametrized
by  a  pair $(Q_0,c)$, where $Q_0$ is a probability measure on $[0,1]$, and $c$ satisfies \eqref{support_Q}:
\be\label{flux_parameter}
f^{Q_0,c}(\rho):=\left\{
\ba{lll}
(p-q)\left(\overline{R}^{Q_0}\right)^{-1}(\rho) & \mbox{if} & \rho<\rho_c(Q_0,c)\\
(p-q)c & \mbox{if} & \rho\geq\rho_c(Q_0,c)
\ea
\right.
\ee
where  $\overline{R}^{Q_0}$ is defined by \eqref{average_afgl_ergodic}, and (recalling conventions \eqref{extension_theta}--\eqref{convention_2})  
\be\label{critical_parameter}
\rho_c(Q_0,c):=\int_{[0,1]}R\left[\frac{c}{a}\right]dQ_0(a)=\int_{[c,1]}R\left[\frac{c}{a}\right]dQ_0(a)
\ee
Then, with definitions \eqref{flux_parameter}--\eqref{critical_parameter}, we can write
\be\label{rewrite_parameter}
f^\alpha=f^{Q_0(\alpha),\inf\alpha},\quad \rho_c^\alpha=\rho_c(Q_0(\alpha),\inf\alpha)
\ee
For a given $Q_0$, the maximal value of $c$ is $C:=\inf{\rm supp}\,Q_0$, cf. \eqref{support_Q}.
For this value of $c$, we denote $f^{Q_0,c}$ by $f^{Q_0}$ and $\rho_c(Q_0,c)$ by $\rho_c(Q_0)$:
\be\label{flux_crit_q}
f^{Q_0}:=f^{Q_0,\inf{\rm supp}\,Q_0},\quad
\rho_c(Q_0):=\rho_c(Q_0,\inf{\rm supp}\,Q_0)
\ee
\begin{remark}\label{remark_critpar}
Since $\rho_c(Q_0,c)$ defined by \eqref{critical_parameter} 
is a nondecreasing function of $c$, $\rho_c(Q_0)$ is the maximal 
critical density one can obtain from $Q_0$.  Note that $\rho_c(Q_0)$ 
may be infinite if the integral in \eqref{critical_parameter} diverges 
for $c=C$, but $\rho_c(Q_0,c)$ is always finite when $c<C$.
\end{remark}
In the context of Example \ref{example_random}, we thus have
\begin{example}\label{example_random_bis}
Let $Q$ be a spatially ergodic probability measure on $\bf A$ with marginal $Q_0$
(for instance, $Q=Q_0^{\otimes\Z}$). Then, for $Q$-almost every $\alpha\in\bf A$,
\be\label{fluxcrit_ex}
f^\alpha=f^{Q_0},\quad
\rho_c^\alpha=\rho_c(Q_0)
\ee
\end{example}
\begin{lemma}\label{lemma_properties_flux}\mbox{}\\ 
 (i)  The functions $\overline{R}^{ Q_0}$  and
$f^{ Q_0,c}$ are increasing and analytic, respectively from   $[0,C)$ to $[0,\overline{R}^{Q_0}(C))$ 
and from $[0,\rho_c(Q_0,c))$ to $[0,(p-q)c)$. \\ 
 (ii) The function  $f^{ Q_0,c}$ is $(p-q)$-Lipschitz. \\
 (iii) The function $f^{ Q_0,c}$ is concave if $g$  satisfies  
\be\label{concave_g}n\mapsto g(n+1)-g(n)\mbox{ is nonincreasing}\ee 
\end{lemma}
\begin{proof}{Lemma}{lemma_properties_flux}
\mbox{}\\ 
 {\em Proof of (i).}   Since $R$ is continuously differentiable on $[0,1)$, by \eqref{average_afgl_ergodic},   
\begin{eqnarray}\nonumber
\left(\overline{R}^{Q_0}\right)'(\beta) & = & 
\int_{[0,1]}
\frac{1}{a}R'\left[\frac{\beta}{a}\right]dQ_0(a)
\label{rbarprime}
\end{eqnarray}
Hence, since the function $R$ defined by \eqref{mean_density} is increasing,  
the function $\overline{R}^{Q_0}$ is increasing,  thus $f^{Q_0,c}$ is increasing on $[0,\rho_c(Q_0,c)]\cap\R$. 
Similarly, \eqref{average_afgl_ergodic} and analyticity of $R$ 
imply analyticity of $\overline{R}^{Q_0}$ and $f^{Q_0,c}$.  \\ \\
 {\em Proof of (ii).}
This  boils down to proving that for any $x\in\Z$,  for any $\rho\in[0,\rho_c)$, 
\[
\left(\overline{R}^{Q_0}\right)^{-1}(\rho)=\int_{\overline{\mathbf{X}}}\alpha(x)g[\eta(x)]d\mu^{\alpha,\rho}(\eta)
\]
is $1$-Lipschitz. Let $\rho\leq\rho'<\rho_c$. 
By \eqref{stoch_inc} and \eqref{invariant_density},  Strassen's Theorem 
(see e.g. \cite{ligbook}) yields a coupling measure 
$\tilde{\mu}(d\eta,d\zeta)$ of $\mu^{\alpha,\rho}$ 
and $\mu^{\alpha,\rho'}$ under which $\eta\leq\zeta$ holds a.s..  
Then,  setting 
\begin{equation}\label{beta-beta'}
\beta=\left(\overline{R}^{Q_0}\right)^{-1}(\rho)\qquad\mbox{and}\qquad
\beta'=\left(\overline{R}^{Q_0}\right)^{-1}(\rho')
\end{equation} 
we have by \eqref{mean_rate}, 
\begin{eqnarray}\nonumber
\left(\overline{R}^{Q_0}\right)^{-1}(\rho')-\left(\overline{R}^{Q_0}\right)^{-1}(\rho)
& = & \alpha(x)\int_{\overline{\mathbf{X}}}[g(\zeta(x))-g(\eta(x))]d\tilde{\mu}(\eta,\zeta)\\
& \leq & \int_{\overline{\mathbf{X}}}[\zeta(x)-\eta(x)]
d\tilde{\mu}(\eta,\zeta)\nonumber\\& = &R\left[
\frac{\beta'}{\alpha(x)}
\right]-R\left[
\frac{\beta}{\alpha(x)}
\right] \label{theaboveineq}
\end{eqnarray}
In the inequality we used $\alpha(x)\leq 1$, \eqref{invariant_density}, 
and the fact that $g$ is nondecreasing and bounded above by $1$,  so that either
$\zeta(x)=\eta(x)$, or $g(\zeta(x))-g(\eta(x))\le 1\le \zeta(x)-\eta(x)$. 
Averaging the inequality \eqref{theaboveineq} over $x=0,\ldots,n$ 
and using \eqref{average_afgl} yields the upper bound
$\overline{R}^{ Q_0}(\beta')-\overline{R}^{ Q_0}(\beta)$  
which is equal to  $\rho'-\rho$ by \eqref{beta-beta'}.  \\ \\
 {\em Proof of (iii). } It is shown in \cite{bs} that \eqref{concave_g} implies concavity of
the flux function for the corresponding {\em homogeneous} zero-range process, or equivalently, 
convexity of $R$. The latter property combined with \eqref{average_afgl_ergodic} 
 implies convexity of $\overline{R}=\overline{R}^{Q_0}$, thus 
concavity of  $f^{Q_0,c}$ defined by \eqref{flux_parameter}. 
\end{proof}
\subsection{A family of deterministic environments}\label{subsec:env}
As pointed out after Example \ref{example_random}, an ergodic environment always yields equality in \eqref{support_Q}.
In this subsection, to illustrate the more general situation where this equality may fail, we define a family of deterministic environments that produces a given pair $(Q_0,c)$ satisfying \eqref{support_Q}, and will serve again later in the section.\\ 
Let 
\be\label{defect_set}
{\mathcal X}:=\{x_n,\,n\in\Z\}
, \quad  {\mathcal Y}:=\{{y}_n,\,n\in\Z\} 
\ee
be  doubly 
infinite increasing $\Z$-valued sequences  (if $c=C$, only $\mathcal Y$ will be used, see \eqref{inversion} below).  
Let also $(\alpha_n)_{n\in\Z}$ be a doubly 
infinite $[c,1]$-valued sequence satisfying 
\be\label{sequence_alpha}\lim_{n\to\pm\infty}\alpha_n=c\ee
For $x\in\Z$, we set 
\begin{eqnarray}
u(x) & := & \sum_{n\in\Z}1_{[{y}_n,{y}_{n+1})}(x)\frac{x- y_n}{ y_{n+1}-y_n}\label{sample_uniform}\\ \nonumber\\
\alpha(x) & := &  F_{Q_0}^{-1}(u(x)) \label{inversion}
{\bf 1}_{\Z\setminus{\mathcal X}}(x)+{\bf 1}_{\{c<C\}}\sum_{n\in\Z}\alpha_n{\bf 1}_{\{ {x}_n\}}(x)
\end{eqnarray}
where $F_{Q_0}(t):=Q_0((-\infty,t])$ denotes the  cumulative distribution function (c.d.f.) 
of $Q_0$, and $F_{Q_0}^{-1}$ its generalized inverse.  The following lemma is established in 
Appendix \ref{app:uniform}.
\begin{lemma}\label{lemma_family}
\mbox{}\\ 
(i) The environment $(\alpha(x))_{x\in\Z}$ satisfies Assumption \ref{assumption_ergo} if and only if  
%
\begin{eqnarray}\label{assumption_afgl_1}
\lim_{n\to\pm\infty}\frac{y_{n+1}}{y_n} & = & 1\\
\label{growth_inc}
\lim_{n\to\pm\infty}\frac{n}{y_n} & = & 0\\
 \label{growth_inc_x}\lim_{n\to\pm\infty}\frac{n}{x_n} & = & 0
\end{eqnarray}
(ii) The environment $(\alpha(x))_{x\in\Z}$ satisfies Assumption \ref{assumption_dense} if and only if $c=C$ and \eqref{assumption_afgl_1} holds, or
$c<C$ and the sequence  $(x_n)_{n\in\Z}$  satisfies condition \eqref{assumption_afgl_0}. 
\end{lemma}
 To prove \textit{(i)} of Lemma \ref{lemma_family}, we must essentially 
 prove that assumption \eqref{growth_inc} is necessary and sufficient for \eqref{sample_uniform} to yield a uniformly distributed 
set of values in the sense that 
\be\label{empirical_uniform}
\lim_{n\to+\infty}\frac{1}{n+1}\sum_{x=0}^n\delta_{u(x)}=
\lim_{n\to+\infty}\frac{1}{n+1}\sum_{x=-n}^0\delta_{u(x)}={\mathcal U}(0,1)
\ee
Indeed, the first term on the r.h.s. of  \eqref{inversion} is nothing 
but the inversion method to generate an arbitrary random variable from a uniform one.
 The environment defined by  \eqref{inversion} has the following interpretation. The first term has fast oscillations that reproduce an ergodic-like behaviour with distribution $Q_0$. This produces the subcritical part of the flux, that is the first line of \eqref{flux_parameter}. When $c<C$, the second term introduces a dense sequence of defects that leads to  \eqref{inf_alpha} and to the supercritical part of the flux, that is the second line of \eqref{flux_parameter}.  Condition \eqref{growth_inc_x} prevents this term from adding an atom at $c$ to $Q_0$.  \\ \\ 
 The following example fulfills the requirements of Lemma \ref{lemma_family}.
\begin{example}\label{Example1} 
$x_n=y_n={\bf 1}_{\{n\neq 0\}}\sgn(n)\lfloor{|n|}^\kappa\rfloor$ 
 with $\kappa>1$,  $c\leq C$. 
\end{example} 
 In the next subsection, we explain why the behaviour of the zero-range process with environment \eqref{inversion} is expected to be different from the one described by Theorem \ref{th_hydro} when the conditions of Lemma \ref{lemma_family} are not fulfilled.
After this, we will always assume these conditions satisfied. 
\subsection{The defect density condition}
\label{subsec:defect}
 Conditions \eqref{assumption_afgl_0} and \eqref{assumption_afgl_1} can be interpreted by saying that
there is no macroscopic separation of points in the corresponding sequence. When these conditions are not satisfied, 
a hidden scaling may emerge,  and the hydrodynamic limit falls outside the scope of Theorem \ref{th_hydro}, although
we  formulate some natural conjectures below. 
The following example  illustrates failure of \eqref{assumption_afgl_0} while \eqref{assumption_afgl_1}--\eqref{growth_inc} hold.
\begin{example}\label{Example3} 
 $x_n={\bf 1}_{\{n\neq 0\}}\sgn(n)\lfloor\kappa^{|n|}\rfloor$,
$y_n={\rm sgn}(n)\lfloor |n|^{\kappa'}\rfloor$ with $\kappa,\kappa'>1$, $c<C$. \end{example} 
 In this example, the set $\mathcal X_N:=N^{-1}\mathcal X$ of rescaled defect 
 locations has a subsequential scaling limit when $N\to+\infty$ with $N\in\mathcal N$, where
\be\label{def_subseq}
{\mathcal N}:=\{\kappa^n:\,n\in\N\setminus\{0\}\}
\ee
Indeed,
\[
\lim_{N\to+\infty,\,N\in\mathcal N}{\bf 1}_{\mathcal X_N}={\bf 1}_B,\quad\mbox{where }B:=\{\pm\kappa^n:\,n\in\Z\}
\]
We then expect the hydrodynamic limit to be given by  
\eqref{conservation_law} outside $B$, and completed by specific boundary 
conditions  on $B$, to indicate that the macroscopic current at these points cannot exceed $c$. 
 These boundary conditions are similar to those introduced 
 in \cite{lan2} to describe the hydrodynamic limit of the 
 totally asymmetric zero-range process with a {\em single} defect.  
The hydrodynamic profile  $\rho(.,t)$ 
at time $t$ is expected to be a measure of the form 
\be\label{form_sol}
\rho(.,t)=\sum_{x\in B}m(x,t)\delta_x+\widetilde{\rho}(.,t)
\ee
where $\widetilde{\rho}$ is a weak entropy solution of 
\eqref{conservation_law} outside $B$, while on $B$, $\rho$
satisfies the boundary conditions
\be\label{bc1}
m(x,t)>0\Rightarrow \frac{\partial m}{\partial t}(x,t)=[f(\widetilde{\rho}(x-,t))-(p-q)c],\quad x\in B
\ee
and $\widetilde{\rho}$ satisfies the boundary conditions
\be\label{bc2}
\widetilde{\rho}(x+,t)=\widetilde{\rho}(x-,t)\wedge\rho_c,\quad x\in B
\ee
These boundary conditions stipulate
that on a time interval where the flux coming from the left exceeds 
$c$ at some $x\in B$, macroscopic condensation occurs
in the form of a growing Dirac mass. 
When the flux comes back below $c$, the condensate starts decreasing 
until either it disappears, or starts growing again if the flux again 
starts exceeding $c$. \\ \\
 The next example satisfies \eqref{assumption_afgl_0} but violates condition \eqref{assumption_afgl_1}. 
\begin{example}\label{Example2} 
 $y_n={\bf 1}_{\{n\neq 0\}}\sgn(n)\lfloor\kappa^{|n|}\rfloor$,
$x_n={\rm sgn}(n)\lfloor |n|^{\kappa'}\rfloor$ with $\kappa,\kappa'>1$,  $c\leq C$.  \end{example} 
 In this case, the environment \eqref{inversion} has a macroscopic profile 
 under the following subsequential scaling limit (with $\mathcal N$ as in \eqref{def_subseq}):
\be
\label{limit_env}
\overline{\alpha}(x):=\lim_{N\to+\infty,\,N\in{\mathcal N}}\alpha(\lfloor Nx\rfloor)
\ee
where
\begin{eqnarray}
\overline{\alpha}(x) & = & 
\sum_{n\in\Z}1_{(\kappa^n,\kappa^{n+1})}(x)F_{Q_0}^{-1}\left(
\frac{x-\kappa^n}{\kappa^{n+1}-\kappa^n}
\right)\nonumber\\
& + & 
\sum_{n\in\Z}1_{(-\kappa^{n+1},-\kappa^{n})}(x)F_{Q_0}^{-1}\left(
\frac{x+\kappa^{n+1}}{\kappa^{n+1}-\kappa^n}
\right) 
\end{eqnarray}
%
%
We then expect the hydrodynamic limit as $N\to+\infty$ in $\mathcal N$ to be given 
by the {\em spatially inhomogeneous} conservation law
\be\label{conservation_inhom}
\partial_t\rho(t,x)+\partial_x\left\{
\overline{\alpha}(x)\min\left[
f_{\rm hom}(\rho(t,x)),(p-q)c
\right]
\right\}=0
\ee
where
\be\label{def_f_hom}
f_{\rm hom}(\rho):=(p-q)R^{-1}(\rho)
\ee
is the flux function of the {\em homogeneous} zero-range process defined 
by \eqref{generator} for $\alpha(.)$ identically equal to $1$.
 The truncation by $(p-q)c$ in \eqref{conservation_inhom} is due to 
 the dense sequence of defects $(x_n)_{n\in\Z}$.
Similar hydrodynamic limits without this term appear in  \cite{b,cek,cr,rez2}.\\ \\
Finally, both conditions \eqref{assumption_afgl_0} and \eqref{assumption_afgl_1} 
may fail simultaneously, as in the following example.
\begin{example}\label{Example5} 
 $x_n=y_n={\bf 1}_{\{n\neq 0\}}\sgn(n)\lfloor\kappa^{|n|}\rfloor$,
 with $\kappa>1$, $c<C$. \end{example} 
Here, we expect the hydrodynamic limit as $N\to+\infty$ in $\mathcal N$ to be of the form 
\[
\partial_t\rho(t,x)+\partial_x\left\{
\overline{\alpha}(x)
f_{\rm hom}(\rho(t,x))
\right\}=0
\]
outside $B$, and the truncation term $(p-q)c$ in \eqref{conservation_inhom} 
to be replaced by
boundary conditions \eqref{bc1}--\eqref{bc2} (where the r.h.s. 
of \eqref{bc1} is now multiplied by $\overline{\alpha}(x)$).\\ \\
{}{}From now on, we assume that conditions \eqref{assumption_afgl_0} and \eqref{assumption_afgl_1}--\eqref{growth_inc_x} of Lemma \ref{lemma_family} are satisfied.
\subsection{The dilute limit}\label{subsec:dilute}
 In general, the subcritical part of the flux is not very explicit, unless specific choices of $Q_0$ make
\eqref{average_afgl_ergodic} computable.
In this subsection, we discuss the  so-called {\em dilute limit} (\cite{kru, bb}),
that is a homogeneous system perturbed by a sequence of defects with vanishing density (but macroscopically dense in the sense of Assumption \ref{assumption_dense}). In this case, the subcritical flux is exactly the flux function $f_{\rm hom}$ (see \eqref{def_f_hom}) of the {\em homogeneous} zero-range process.
One way to obtain this limit is to consider
the special case 
$Q_0=\delta_1$  (hence $C=1$) 
of  \eqref{sample_uniform}--\eqref{inversion}, that is
\be\label{example_env}
\alpha(x)=\left\{
\ba{lll}
\alpha_n & \mbox{if} & x=x_n,\quad n\in\Z\\ 
1 & \mbox{if} & x\not\in\mathcal X
\ea
\right.
\ee
Then, for every ${\beta}\in[0,c)$, the limits in \eqref{average_afgl} 
exist and are similar to the limit obtained for  the  homogeneous zero-range process: 
%
\be\label{afgl_hom}
\overline{R}({\beta})=R({\beta})
\ee
Therefore,  by \eqref{other_def_critical},  the critical density is given by
\be\label{dilute_critical}
\rho_c=\overline{R}(c-)=R(c)
\ee
\begin{remark}\label{remark_point_out_00}
 Recall now  Remark \ref{remark_point_out_0}: 
if we wanted here to define $\rho_c$ using \eqref{average_afgl} for $\beta=c$, 
this would yield 
\begin{eqnarray}\nonumber
\overline{R}(c) & = &\lim_{n\to+\infty}\frac{1}{x_n+1}\sum_{k=0}^n R\left[\frac{c}{\alpha_k}\right]
 + R(c) \\
 & =  &
\lim_{n\to+\infty}\frac{1}{1-x_{-n}}\sum_{k=0}^n R\left[\frac{c}{\alpha_{-k}}\right]
 + R(c) \label{bad_def_critical}
\end{eqnarray}
The above limits may not exist, or exist and not coincide, 
and even if they do, their common value
depends on the  respective speeds  of convergence of the sequence 
$(\alpha_n)_{n\in\Z}$ to $c$ and 
 of  the sequence  $(n/x_n)_{n\in\Z}$ to  $0$  (recall that $R(1)=+\infty$  and that
$n/x_n$ vanishes by \eqref{growth_inc_x}).  
It is possible to tune these speeds 
so as to obtain any prescribed finite or infinite limits 
in \eqref{bad_def_critical}. In particular, if $\alpha_n$ 
has constant value $c$, the two limits in \eqref{bad_def_critical} 
are equal to $+\infty$, that is different from \eqref{dilute_critical}. 
\end{remark}
The flux function defined by \eqref{def_flux}--\eqref{extension_flux} becomes here
 (recall \eqref{def_f_hom}; the index $d$  stands for ``dilute'')
\be\label{dilute_flux}
f_d(\rho):=\left\{
\ba{lll}
f_{\rm hom}(\rho)
& \mbox{if} & \rho<\rho_c\\ 
 (p-q) c & \mbox{if} & \rho\geq\rho_c
\ea
\right\}=f_{\rm hom}(\rho)\wedge  (p-q) c
\ee
The dilute limit \eqref{dilute_flux} can be understood 
intuitively as follows. Due to 
\eqref{growth_inc},  slow sites 
are very rare, hence the system exhibits long homogeneous stretches 
where it behaves as a homogeneous process. Thus the memory 
of slow sites is only retained by the flux truncation, but not 
by the shape of the flux function prior to truncation. \\ \\
This phenomenon was pointed out in \cite{kru}  for our model and for the totally 
asymmetric simple exclusion process with site disorder,
and established in \cite{bb} in the latter case.  \\ \\
\textbf{Dilute limit from a random environment.} 
A different way to recover the dilute limit (which in fact corresponds to 
\cite{kru,bb}) is a double limit for an i.i.d. environment
where the probability of a slow site is $\varepsilon$, and $\varepsilon\to 0$ 
after the scaling parameter. 
Let $Q_0$ be a probability measure  on $[c,1]$,  and define 
\be\label{def_qeps}
Q_0^\varepsilon:=(1-\varepsilon)\delta_1+\varepsilon Q_0
\ee
Referring to \eqref{average_afgl_ergodic}  and \eqref{flux_crit_q},  we shall more simply write
$\overline{R}^{\,\,\varepsilon}$ for $\overline{R}^{Q_0^\varepsilon}$, $\rho_c(\varepsilon)$ for $\rho_c(Q_0^\varepsilon)$,  and $f^\varepsilon$ for $f^{Q_0^\varepsilon}$.  
 Recall that, by Example \ref{example_random_bis}, $f^\varepsilon$ is the flux produced by any random environment whose distribution
$Q^\varepsilon$ is ergodic with marginal $Q_0^\varepsilon$ (for instance, the product measure with marginal $Q_0^\varepsilon$ corresponding to i.i.d. environment). 
It follows from \eqref{def_qeps} that
\begin{eqnarray}\label{dilute_ergodic_2}
\overline{R}^{\,\,\varepsilon}({\beta}) & = & (1-\varepsilon)R({\beta})+\varepsilon
\overline{R}^{Q_0}(\beta),\quad \beta\in(0,c)\label{dilute_ergodic_3}\\
\rho_c(\varepsilon) & = & (1-\varepsilon)R(c)+\varepsilon\rho_c(Q_0)\label{dilute_ergodic_3}
\end{eqnarray}
Thus, if $\rho_c(Q_0)<+\infty$, we have 
\begin{eqnarray}
\label{dilute_critical_2}
\lim_{\varepsilon\to 0}\overline{R}^{\,\,\varepsilon}({\beta}) 
& = & R({\beta}),\quad{\beta}\in[0,c)\\
\label{dilute_critical_3}
\lim_{\varepsilon\to 0}\rho_c(\varepsilon) & = & R(c)=\rho_c\\
\label{dilute_critical_4}
\lim_{\varepsilon\to 0}f^\varepsilon(\rho) & = & f_d(\rho)
\end{eqnarray}
\subsection{Supercritical entropy solutions}\label{subsec:super_sol}
We now describe the consequences of the flat line \eqref{extension_flux} 
on the behaviour of entropy solutions.
This is best understood through the analysis of the so-called {\em Riemann} problem, 
{\em i.e.} the Cauchy problem for particular
initial data of the form  \eqref{eq:rie}, 
for which entropy solutions can be computed explicitely,  see \eqref{entropy_riemann}. 
 In the following proposition,  we analyze the Riemann solution when 
\be\label{analyze} +\infty>\lambda\geq\rho_c>\rho\ee
\begin{proposition}\label{cor_riemann}
 Assume \eqref{analyze}.  Let
\be\label{critical_speed}
 v_c(\rho):=\inf_{r\in[\rho,\rho_c)}\frac{f(\rho_c)-f(r)}{\rho_c-r}  =
\inf_{r\in[\rho,\rho_c)}\frac{\widehat{f}(\rho_c)-\widehat{f}(r)}{\rho_c-r}
=\widehat{f}\,'(\rho_c-) 
\ee
 where $\widehat{f}$ denotes the concave envelope of $f$ on $[\rho,\rho_c]$. 
 In particular, if $f$ is concave,
\be\label{critical_speed_concave}
v_c(\rho)=f'(\rho_c-)=\left\{\int_{[c,1]}\frac{1}{a}R'\left[
\frac{c}{a}
\right]
dQ_0(a)\right\}^{-1}
\ee
Then, for every $t>0$, we have
\begin{eqnarray}
\label{riemann_left}
R_{\lambda,\rho}(x,t) & = & \lambda,\quad\forall x<0\\ 
R_{\lambda,\rho}(x,t) & = & R_{\rho_c,\rho}( x,t ),\quad\forall x>0 \label{riemann_right}\\ 
 \lim_{t\to+\infty}R_{\lambda,\rho}(x,t) & = & \rho_c,\quad\forall x\geq 0\label{endupcrit}\\
R_{\lambda,\rho}(0+,t) & = & \rho_c\label{critical_origin}\\ 
R_{\lambda,\rho}(x,t) & = & \rho_c,\quad\forall x\in(0,tv_c(\rho))\label{front}\\
R_{\lambda,\rho}(x,t) & < & \rho_c,\quad\forall x>tv_c(\rho)\label{nofront}
\end{eqnarray}
\end{proposition}
 We prove this  proposition in the next subsection,  but we first comment on its interpretation
and give examples. 
Property \eqref{riemann_left} states that
the initial  constant  density 
is not modified to the left of the origin.
This is not related to phase transition, but only to the fact that $f$ 
is nondecreasing, hence characteristic velocities are always nonnegative. 
Properties \eqref{endupcrit}--\eqref{front} are signatures of the phase transition.
They express the fact that, regardless of the supercritical value on the left side,
supercritical densities are blocked, and
the right side is dominated by the critical density.
In particular, \eqref{front}--\eqref{nofront} state that
a front of critical 
density propagates to the right from the origin  at speed $v_c(\rho)>0$ if $v_c(\rho)>0$. 
The positivity of $v_c(\rho)$ is thus an interesting property to investigate.
In particular, \eqref{critical_speed_concave} shows (similarly to Remark \ref{remark_critpar}) 
that $v_c(\rho)>0$ if $c<C$, whereas if $c=C$, $v_c(\rho)$ may be infinite if the integral in 
\eqref{critical_speed_concave} diverges at $C$.\\ \\
To be more explicit, let us examine the following examples, 
where for simplicity we assume $p=1$ and $q=0$. 
\begin{example}\label{example_front_dilute}
We consider the $M/M/1$ queues in series, 
that is $g(n)=n\wedge 1$, in 
the dilute limit  \eqref{example_env}.
\end{example}
With this choice of $g$,
\eqref{mean_density} and \eqref{def_f_hom} write
\be\label{mm1}
R(\beta)  =  \frac{\beta}{1-\beta},\quad
f_{\rm hom}(\rho) = \frac{\rho}{1+\rho}\label{mm1_fhom}
\ee
Recall that in the dilute limit $f$ is given by $f_d$ defined in \eqref{dilute_flux}. The latter, in view of \eqref{mm1_fhom}, writes 
\be\label{mm1_dilute}
 f_d(\rho)=\left[\frac{\rho}{1+\rho}\right]\wedge c=\left\{
\ba{lll}
\displaystyle{\frac{\rho}{1+\rho} } & \mbox{if} & \displaystyle{\rho<\rho_c:=\frac{c}{1-c}}\\ 
c & \mbox{if} & \rho\geq\rho_c
\ea
\right.
\ee 
 Since $f_d$ defined by \eqref{mm1_dilute} is concave, \eqref{critical_speed} yields 
\be\label{dilute_cspeed}
v_c(\rho) =f'_{\rm hom}(\rho_c^-)=(1-c)^2
\ee
 The next example exhibits a transition between $v_c(\rho)=0$ and $v_c(\rho)>0$. 
\begin{example}\label{following_example}
We consider the $M/M/1$ queues in series, 
that is $g(n)=n\wedge 1$, and $c=C$ (for instance, coming from an ergodic environment 
with marginal $Q_0$, cf. Example \ref{example_random}).
\end{example}
%
%
Given \eqref{mm1_fhom}, 
the critical density is $\rho_c(Q_0)$ defined by \eqref{flux_crit_q} and \eqref{critical_parameter}, hence 
\be\label{pt_mm1}
\rho_c(Q_0)=\int_{[c,1]}\frac{c}{a-c}dQ_0(a)
\ee
 By \textit{(iii)} of Lemma \ref{lemma_properties_flux}, 
$f$ is concave. It follows from  \eqref{critical_speed_concave} that
\be\label{nonzerospeed_mm1}
v_c(\rho)=\left\{\int_{[c,1]}\frac{a}{(a-c)^2}dQ_0(a)\right\}^{-1}  
\ee 
{\em A critical exponent.}  Assume now  that under $Q_0$, $\alpha(0)$ has a density $q_0$ on $( c,1]$ such that
\be\label{density_env}
q_0( t)
\stackrel{ t\to c}{\sim}a( t-c)^{\kappa} 
\ee
for some constants $a>0$ and $\kappa>-1$.
Then  $\rho_c(Q_0)<+\infty$ is equivalent to $\kappa>0$ and $v_c(\rho)>0$ is equivalent to $\kappa>1$. 
\subsection{Proof of Proposition \ref{cor_riemann}\label{subsec:proof_cor}}
We conclude this section with the proof of Proposition \ref{cor_riemann}. 
For this proof, we recall the following construction and result 
for the Riemann entropy solution \eqref{entropy_riemann}, 
which will also be useful in Section \ref{sec_proof_hydro}.
Let $\lambda,\rho,v\in\R$. If $\lambda\leq\rho$, we set
\begin{eqnarray}
\nonumber
{\mathcal G}_{\lambda,\rho}(v) & := & 
\inf\left\{
f(r)-vr:\,r\in[\lambda,\rho]
\right\} \\
&=&
\inf\left\{
(p-q)\theta\wedge c-v\overline{R}(\theta):
\,\theta\in[\overline{R}^{\,\,-1}(\lambda),\overline{R}^{\,\,-1}(\rho)]
\right\}\label{limiting_current_inf} \\
\label{density_inf} h(v) & := & {\rm argmin}\left\{
f(r)-vr:\,r\in[\lambda,\rho]
\right\}
\end{eqnarray}
If $\lambda\geq\rho$, we set 
\begin{eqnarray}
\nonumber
{\mathcal G}_{\lambda,\rho}(v) & := &
\sup\left\{
f(r)-vr:\,r\in[\rho,\lambda]
\right\}  \\
&=&
\sup\left\{
(p-q)\theta\wedge c-v\overline{R}(\theta):\,\theta\in
[\overline{R}^{\,\,-1}(\rho),\overline{R}^{\,\,-1}(\lambda)]
\right\}\label{limiting_current_sup} \\
\label{density_sup}
h(v) & := & {\rm argmax}\left\{
f(r)-vr:\,r\in[\rho,\lambda]
\right\}
\end{eqnarray}
Note that $h(v)$ is a priori well defined if and only if the infimum in \eqref{limiting_current_inf}, or 
the supremum in \eqref{limiting_current_sup}, is uniquely achieved.
\begin{proposition}\label{proposition_1} (\cite{bgrs5})
\mbox{}\\ 
o) If $\lambda<\rho$ (resp. $\lambda>\rho$), $h=(\widehat{f}')^{-1}$, 
 where $\widehat{f}$ denotes  the  convex (resp. concave)  envelope of $f$ 
 on $[\lambda,\rho]$ (resp. $[\rho,\lambda]$). \\ 
i) There exists an at most countable set $\Sigma(\lambda,\rho)$ 
such that the infimum in \eqref{limiting_current_inf}, or
the supremum in \eqref{limiting_current_sup},
is uniquely achieved for every $v\in\R\setminus\Sigma(\lambda,\rho)$.\\ 
ii) The function $h$ thus defined outside $\Sigma(\lambda,\rho)$ 
can be extended to $\R$ into a function (still denoted by $h$) 
that is nondecreasing if $\lambda\leq\rho$, nonincreasing if $\lambda\geq\rho$.\\ 
iii) Let $v\in\Sigma(\lambda,\rho)$. If $\lambda\leq\rho$, 
$h(v-)$ is the smallest and $h(v+)$ the largest  minimizer in \eqref{limiting_current_inf}.
If $\lambda\geq\rho$, $h(v-)$ is the largest and 
$h(v+)$ the smallest  maximizer in \eqref{limiting_current_sup}.\\ 
iv) For every $v,w\in\R$,
\be
\int_v^w h(u)du  =  {\mathcal G}_{\lambda,\rho}(w)-{\mathcal G}_{\lambda,\rho}(v)\label{limiting_current}
\ee 
(v) The function
\be\label{def_riemann_sol}
R_{\lambda,\rho}(x,t):=h(x/t)
\ee
is the unique entropy solution to \eqref{conservation_law} 
with Cauchy datum $R_{\lambda,\rho}(.)$ defined 
by \eqref{eq:rie}.
\end{proposition}
We can now prove Proposition \ref{cor_riemann}.\\
\begin{proof}{Proposition}{cor_riemann}
We first prove the equalities in \eqref{critical_speed}. 
Since $\widehat{f}(\rho_c)=f(\rho_c)$ 
and $\widehat{f}\geq f$, the second member of \eqref{critical_speed} cannot 
be smaller than the third one. The equality between the third and fourth 
quantities follows from concavity of $\widehat{f}$.
Assume there exists
$r\in[\rho,\rho_c)$ such that $\widehat{f}$ is linear on $[r,\rho_c]$, 
and let $r_0$ be the infimum of such values $r$. Then $\widehat{f}(r_0)=f(r_0)$, and
\[
v_c(\rho)\leq \frac{f(\rho_c)-f(r_0)}{\rho_c-r_0}=\frac{\widehat{f}(\rho_c)-\widehat{f}(r_0)}{\rho_c-r_0}
=\inf_{r\in[\rho,\rho_c)}\frac{\widehat{f}(\rho_c)-\widehat{f}(r)}{\rho_c-r}
\]
where the last  equality  follows from concavity of $\widehat{f}$. 
Thus the second and third members of \eqref{critical_speed} coincide, 
and the infimum is achieved for $r=r_0$. 
Assume now that there exists
no $r\in[\rho,\rho_c)$ such that $\widehat{f}$ is linear on $[r,\rho_c]$. 
Then there exists a sequence $(r_n)_{n\in\N}$ converging to $\rho_c$ 
such that $\widehat{f}(r_n)=f(r_n)$, for otherwise one would have 
$\widehat{f}>f$, thus $\widehat{f}$ linear, on a left neighborhood of $\rho_c$.
Then
\be\label{no_linear}
v_c(\rho)\leq \lim_{n\to +\infty}\frac{{f}(\rho_c)-{f}(r_n)}{\rho_c-r_n}=
\lim_{n\to +\infty}\frac{\widehat{f}(\rho_c)-\widehat{f}(r_n)}{\rho_c-r_n}=\widehat{f}\,'(\rho_c-)
\ee
Thus the above inequality is an equality.\\ \\
 For the sequel of the proof, recall that, by \eqref{density_sup},
\be\label{recall_sup}
R_{\lambda,\rho}(x,t)={\rm argmax}_{\theta\in[\lambda,\rho]}\left[f(\theta)-\frac{x}{t}\theta\right]
\ee
{\em Proof of \eqref{riemann_left}.} Since $f$ is nondecreasing and we take $x<0$, 
\[
f(\lambda)-\frac{x}{t}\lambda>f(r)-\frac{x}{t}r
\]
for all $x<0$ and $r\in[\rho,\lambda]$. Thus \eqref{riemann_left} 
follows  from \eqref{recall_sup} and \textit{(v)} of Proposition \ref{proposition_1}. \\ \\
{\em Proof of \eqref{riemann_right}.}
For any $r>\rho_c$, we have $f(r)=f(\rho_c)$, thus for $x>0$,
\[ f(\rho_c)-\frac{x}{t}\rho_c>f(r)-\frac{x}{t}r \]
whence the result.\\ \\
{\em Proof of \eqref{critical_origin}. } Let $\varepsilon>0$ and $r\in[\rho,\rho_c-\varepsilon]$.
Then
\[
f(r)-\frac{x}{t}r<f(\rho_c)-\frac{x}{t}\rho_c
\]
as soon as 
\be\label{as_soon_as}
x<t\inf_{r\in[\rho,\rho_c-\varepsilon]}\frac{f(\rho_c)-f(r)}{\rho_c-r}=:tv_c^\varepsilon(\rho)
\ee
Thus $R_{\lambda,\rho}(x,t)>\rho_c-\varepsilon$ for $x$ satisfying \eqref{as_soon_as}. 
 By \eqref{riemann_right} and \eqref{recall_sup},
\be\label{max_principle}
R_{\lambda,\rho}(x,t)\leq\rho_c,\quad\forall x>0
\ee
Finally, $v_c^\varepsilon(\rho)>0$ 
because $f$ is strictly  increasing and continuous (recall Lemma \ref{lemma_properties_flux}). 
 Hence, $R_{\lambda,\rho}(x,t)>\rho_c-\varepsilon$ for $x<tv_c^\varepsilon(\rho)$.   \\ \\
{\em Proof of \eqref{front}. } If $0<x<tv_c(\rho)$, 
by definition \eqref{critical_speed} of $v_c(\rho)$, we have
\[
f(r)-\frac{x}{t}r<f(\rho_c)-\frac{x}{t}\rho_c
\]
for any $r\in[\rho,\rho_c)$. This implies $R_{\lambda,\rho}(x,t)\geq\rho_c$. 
 Recalling \eqref{max_principle},  
the proof is complete.\\ \\
{\em Proof of \eqref{nofront}. } If $x>tv_c(\rho)$, by definition \eqref{critical_speed} of $v_c(\rho)$,
there exists $r\in[\rho,\rho_c)$ such that
\[
f(r)-\frac{x}{t}r>f(\rho_c)-\frac{x}{t}\rho_c
\]
Thus $R_{\lambda,\rho}(x,t)\neq\rho_c$, hence,  by \eqref{max_principle}, 
$R_{\lambda,\rho}(x,t)<\rho_c$.\\ \\
{\em Proof of \eqref{endupcrit}.} By \eqref{recall_sup}, any subsequential limit $R_\infty$ of $R_{\lambda,\rho}(x,t)$ as $t\to+\infty$ must be a maximizer of $f$ on $[\rho,\lambda]$. By \eqref{max_principle},
$R_\infty\leq\rho_c$. Thus, $R_\infty=\rho_c$. 
\end{proof}
\section{ Proof of Theorem \ref{th_hydro}}\label{sec_proof_hydro}
 We hereafter develop the proof of Theorem \ref{th_hydro} along the lines explained after 
the statement of the Theorem. Precisely, in Subsection \ref{sec_mat}, 
we recall the Harris construction of the process and state some useful properties of the current.
In Subsection \ref{subsec:reduc}, we reduce the
problem of general hydrodynamics to that of Riemann hydrodynamics (Corollary \ref{cor_hydro_riemann}) via
the study of the asymptotic current in such systems (Proposition \ref{prop_hydro_riemann}). 
To this end, we construct microscopic Riemann states by means of {\em pseudo-equilibrium} states.
In Subsection \ref{subsec:interface},
we introduce the interface process and state a scaling limit result 
for this process (Proposition \ref{prop_hydro_interface}), that will be 
proved in parallel to Proposition \ref{prop_hydro_riemann}.
Finally, Subsection \ref{subsec:proof_riemann} is the core of the proof of 
Propositions \ref{prop_hydro_riemann} and \ref{prop_hydro_interface}. 
A key ingredient of this proof is the study of pseudo-equilibrium current 
and density, that is stated as Proposition \ref{lemma_inv_meas_2} 
and proved in Subsection \ref{subsec:current_pseudo}.  
\subsection{Preliminary material}\label{sec_mat}
We first recall some definitions and preliminary results on the graphical construction and currents from
\cite{bmrs1,bmrs2}.
\subsubsection{Harris construction and coupling}\label{subsec_harris}\label{subsec:Harris}
We introduce a probability space $(\Omega,\mathcal F,\Prob)$, whose
generic element $\omega$ - called a Harris system (\cite{har}) - of $\Omega$  
is a locally finite point measure of the form
\be\label{def_omega}
\omega(dt,dx,du,dz)=\sum_{n\in\N}\delta_{(T_n,X_n,U_n,Z_n)}
\ee
 on   $(0,+\infty)\times\Z\times(0,1)\times\{-1,1\}$,   
 where $\delta_{(.)}$ denotes Dirac measure, and $(T_n,X_n,U_n,Z_n)_{n\in\N}$ is a 
 $(0,+\infty)\times\Z\times(0,1)\times\{-1,1\}$-valued sequence.  
We denote by $\Exp$ the expectation corresponding to the probability measure $\Prob$. 
Under  $\Prob$, $\omega$ is a Poisson measure with 
intensity 
\be\label{def_intensity}
dtdx\indicator{[0,1]}(u)du\, p(z)dz
\ee
We write $(t,x,u,z)\in\omega$ when $\omega(\{(t,x,u,z)\})=1$, and
we also say that $(t,x,u,z)$ is a potential jump event. 
On  $(\Omega,\mathcal F,\Prob)$,  a c\`adl\`ag process $(\eta_t^\alpha)_{t\geq 0}$ 
with generator \eqref{generator} and initial configuration $\eta_0$ can be constructed
in a unique way  (see \cite[Appendix B]{bmrs2})  so that 
\be\label{rule_1}
\forall(s,x,v,z)\in\omega,\quad
v\leq \alpha(x)g\left[\eta^\alpha_{s-}(x)\right]  \Rightarrow \eta^\alpha_s=(\eta^\alpha_{s-})^{x,x+z}
\ee
and, for all $x\in\Z$ and $0\leq s\leq s'$,
\begin{eqnarray}\label{rule_2}
&&\omega\left(
(s,s']\times E_x
\right)=0\Rightarrow\forall t\in(s,s'],\,\eta_t(x)=\eta_s(x) \\
\nonumber
\mbox{where}&& 
E_x:=\{(y,u,z)\in\Z\times(0,1)\times\{-1,1\}:\,
x\in\{y,y+z\}
\}
\end{eqnarray}
(note that the inequality in \eqref{rule_1} implies $\eta_{t-}^\alpha(x)>0$, 
cf. \eqref{properties_g}, thus $(\eta^\alpha_{t-})^{x,x+z}$ is well-defined). 
Equation \eqref{rule_1} says when a potential jump event gives rise to an actual jump,
 while \eqref{rule_2} states that no jump ever occurs outside potential jump events.
This process defines a random flow
\be \label{unique_mapping_0} (\alpha,\eta_0,t)\in{\mathbf A}\times{\bf
X}\times\R^+\mapsto\eta_t^\alpha= \eta_t(\alpha,\eta_0,\omega) \in{\mathbf X} \ee
 In particular, this flow allows us to couple 
an arbitrary number of processes with generator \eqref{generator}, 
corresponding to different values of $\eta_0$, by using the same 
Poisson measure $\omega$ for each of them.  Since $g$ is nondecreasing, 
the update rule \eqref{rule_1} implies that 
\be \label{attractive_1}
(\alpha,\eta_0,t)\mapsto\eta_t(\alpha,\eta_0,\omega) \mbox{ is
nondecreasing w.r.t. } \eta_0\ee
It follows that the process is \textit{completely monotone}, 
and thus attractive (see \cite[Subsection 3.1]{bgrs5}). 
For instance, the coupling of two processes $(\eta_t^{\alpha})_{t\geq 0}$ 
and $(\zeta_t^{\alpha})_{t\geq 0}$
behaves as follows. Assume $\omega(\{(t,x,u,z)\})=1$
and 
that (without loss of generality)  
$\eta^\alpha_{t-}(x)\leq\zeta^\alpha_{t-}(x)$, so that (since $g$ is nondecreasing) 
$g\left(\eta^\alpha_{t-}(x)\right)\leq g\left(\zeta^\alpha_{t-}(x)\right)$. 
Then the following jumps  from $x$ to $x+z$ occur at time $t$:\\ 
\textit{(J1)} If $u\leq\alpha(x)g\left(\eta^\alpha_{t-}(x)\right)$, an $\eta$ 
and a $\zeta$ particle simultaneously jump.\\ 
\textit{(J2)} If $\alpha(x)g\left(\eta^\alpha_{t-}(x)\right)<u\leq\alpha(x)g\left(\zeta^\alpha_{t-}(x)\right)$, 
a $\zeta$ particle  alone jumps.\\ 
\textit{(J3)} If $\alpha(x)g\left(\zeta^\alpha_{t-}(x)\right)<u$, nothing happens.\\ \\
The above dynamics implies that $\left(\eta^\alpha_t,\zeta^\alpha_t\right)_{t\geq 0}$ 
is a Markov process on $\overline{\mathbf X}^2$ with generator
\begin{eqnarray}
\nonumber\widetilde{L}^\alpha f(\eta,\zeta) & = & 
\sum_{x,y\in\Z}\alpha(x)p(y-x)\left(g(\eta(x))\wedge g(\zeta(x))\right)\left[
f\left(\eta^{x,y},\zeta^{x,y}\right)-f(\eta,\zeta)
\right]\\
& + & \sum_{x,y\in\Z}\alpha(x)p(y-x)[g(\eta(x))-g(\zeta(x))]^+\left[
f\left(\eta^{x,y},\zeta\right)-f(\eta,\zeta)
\right]\nonumber\\
& + & \sum_{x,y\in\Z}\alpha(x)p(y-x)[g(\zeta(x))-g(\eta(x))]^+\left[
f\left(\eta,\zeta^{x,y}\right)-f(\eta,\zeta)
\right]\nonumber\\&&\label{coupled_generator}
\end{eqnarray}   
\subsubsection{Currents}\label{subsec_currents}
Let $x_.=(x_s)_{s\geq 0}$ denote a $\Z$-valued piecewise constant c\`adl\`ag   path such 
that $\vert x_s-x_{s-}\vert\leq 1$ for all $s\geq 0$. In the sequel we shall  use  
paths $(x_.)$ independent of the Harris system used for the particle dynamics,
hence we may assume that $x_.$ has no jump time in common with the latter.  
We denote by 
$\Gamma_{x_.}^\alpha(\tau,t,\eta)$ 
the rightward current across the path $x_.$ 
in the time interval $(\tau,t]$ in the process  $(\eta_s^\alpha)_{s\geq \tau}$  starting from $\eta$ 
in environment $\alpha$, that is  the sum of two contributions. The contribution of particle jumps is
 the number of times a particle jumps from $x_{s-}$ to $x_{s-}+1$ (for $\tau<s\le t$), 
 minus the number of times a particle jumps from $x_{s-}+1$ to $x_{s-}$. 
 The contribution of path motion is obtained by summing over jump times
$s$ of the path, a quantity equal to the number of particles at $x_{s-}$ if the jump is to the left, or
 minus  the number of particles at $x_{s-}+1$ if the jump is to the right. 
 If  
\be\label{finite_right}\sum_{x>x_\tau}\eta(x)<+\infty\ee
 we also have
 \be\label{current}
 \Gamma^\alpha_{x_.}(\tau,t,\eta)=\sum_{x>x_t}\eta_t^\alpha(x)-\sum_{x>x_\tau}\eta(x)
 \ee
It follows from \eqref{current}  that if $x_.$ and $y_.$ are two paths  as above,  then
\be\label{difference_currents}
\Gamma_{y_.}^\alpha(\tau,t,\eta)-\Gamma_{x_.}^\alpha(\tau,t,\eta)
=-\sum_{x=x_t+1}^{y_t}\eta_t(x)+\sum_{x=x_\tau+1}^{y_\tau}\eta(x)
\ee
with the convention 
$\sum_{x=a}^{b}:=0$ if $a>b$.
 Formula \eqref{difference_currents} remains valid even if \eqref{finite_right} does not hold.\\ 
For $x_0\in\Z$, we shall write $\Gamma^\alpha_{x_0}$ for the current across the 
fixed site $x_0$; that is, $\Gamma^\alpha_{x_0}(\tau,t,\eta):=\Gamma^{\alpha}_{x_.}(\tau,t,\eta)$,
where $x_.$ is the constant path defined by $x_t=x_0$ for all  $t\geq \tau$. 
If $\tau=0$, we simply write $\Gamma_{x_.}^\alpha(t,\eta)$ or $\Gamma_{x_0}^\alpha(t,\eta)$ instead of
$\Gamma_{x_.}^\alpha(0,t,\eta)$ or  $\Gamma_{x_0}^\alpha(0,t,\eta)$.
 It follows from the above definition of the current that, for every $x\in\Z$,
\be\label{standard_current}
\qquad\Exp\left[\Gamma^\alpha_x(\tau,t,\eta)\right]=\Exp\left\{\int_\tau^t \left\{
p\alpha(x)g[\eta^\alpha_s(x)]-q\alpha(x+1)g[\eta^\alpha_s(x+1)]
\right\}ds\right\}
\ee
\\
The following results will be important tools to compare currents.
Let us couple two processes $(\zeta_t)_{t\geq 0}$ and $(\zeta'_t)_{t\geq 0}$ 
through the Harris construction,  with $x_.=(x_s)_{s\geq 0}$ as above.  
\begin{lemma}\label{lemma_finite_prop}
For each  $V > 1$, there exists $b=b(V) >  0$  such that for large enough $t$,
 if $ \zeta_0 $ and $\zeta'_ 0 $ agree on an interval $(x,y)$, then, outside probability $e^{-bt}$, 
\[
\zeta_s(u) = \zeta'_s(u) \quad\mbox{for all }0 \leq s \leq t \mbox{ and } u\in (x+Vt,y-Vt)
\]
\end{lemma}
Lemma \ref{lemma_finite_prop} is a version of {\em finite propagation property},
proved in  \cite{bmrs1} as well as Corollary \ref{corollary_consequence} below.
Next lemma is an
adaptation of \cite[Corollary 4.2]{bmrs2}. 
\begin{lemma}
\label{lemma_current}
For a particle configuration  $\zeta\in\overline{\mathbf{X}}$  and a site $x_0\in \Z$, we define
\[ 
F_{x_0}(x,\zeta):=\left\{
\ba{lll}
\sum_{y=1+x_0}^{x}\zeta(y) & \mbox{if} & x>x_0\\ \\
-\sum_{y=x}^{x_0}\zeta(y) & \mbox{if} & x\leq x_0
\ea
\right.
\] 
 For any $0\le t_0\le t$,  define
$\displaystyle{x_t^M= \sup_{s\in[t_0,t]} x_s}$ and
$\displaystyle{x_t^m= \inf_{s\in[t_0,t]} x_s}$.
Let $\zeta_0\in\overline{\mathbf{X}},\zeta'_0\in{\mathbf{X}}$. 
Then,  given  $V>1$, 
\begin{eqnarray*}
&& 0\vee\sup_{x\in [\min(x_0,x_t^m)-V(t-t_0), \max(x_0,x_t^M)+1+V(t-t_0)]}
\left[F_{x_0}(x,\zeta_0)-F_{x_0}(x,\zeta'_0)\right]\\
& \geq & \Gamma^\alpha_{x_.}(t_0,t,\zeta'_0)-\Gamma^\alpha_{x_.}(t_0,t,\zeta_0)
\end{eqnarray*}
with 
probability greater than $1-Ce^{-(t-t_0)/C}$, where $C$ is a positive constant depending only on $V$.
\end{lemma}  
\begin{corollary}\label{corollary_consequence}
For $y\in\Z$, define the configuration 
\be\label{conf-max}
\eta^{*,y}:=(+\infty)\indicator{(-\infty,y]\cap\Z}
\ee
Then,  for any $\zeta\in\overline{\mathbf{X}}$, 
%
\[ 
%
\Gamma^\alpha_y(t,\zeta)\leq\Gamma^\alpha_y(t,\eta^{*,y})
\] 
\end{corollary} 
 Finally,  the following result (see \cite[Proposition 4.1]{bmrs2}) is concerned  
with the asymptotic current produced by a source-like initial condition.
\begin{proposition}\label{current_source} 
Assume $x_t$ is such that $\lim_{t\to+\infty}t^{-1}x_t$ 
exists.  Let $\eta^{\alpha,t}_0:=\eta^{*,x_t}$, see \eqref{conf-max}.
Then 
\begin{eqnarray*}
\limsup_{t\to\infty}\left\{\Exp\left|
t^{-1}\Gamma^\alpha_{x_t}(t,\eta^{*,x_t})-(p-q)c
\right| - p[\alpha(x_t)-c]\right\}
& \leq & 0 \label{upperbound_current_source}
\end{eqnarray*}
\end{proposition}
\subsection{Reduction to the Riemann problem}\label{subsec:reduc}
Precisely, we shall use
the following definition and theorem from \cite{bgrs}. Let  $\eta,\xi\in\mathbf{X}$ be two particle configurations with finite mass to the left, that is
\[
\max\left(
\sum_{x\leq 0}\eta(x),
\sum_{x\leq 0}\xi(x)
\right)<+\infty
\]
we define
\[
\Delta(\eta,\xi):=\sup_{x\in\Z}\left|\sum_{y\leq x}[\eta( y)-\xi( y)]\right|
\]
\begin{definition}(\cite[Definition 3.1]{bgrs}).\label{def_macrostab}
The process defined by \eqref{generator} is {\em macroscopically stable} if it enjoys the following property. Let  $(\eta^N_0)_{N\in\N\setminus\{0\}}$ and $(\xi^N_0)_{N\in\N\setminus\{0\}}$ be any two sequences of initial configurations with uniformly bounded mass in the sense
\be\label{assumption_mass}
\sup_{N\in\N\setminus\{0\}} N^{-1}\max\left(
\sum_{x\in\Z}\eta^N_0(x),
\sum_{x\in\Z}\xi^N_0(x)
\right)<+\infty
\ee
Then, for every $t>0$, it holds that
\be\label{decrease_delta}
N^{-1}
\Delta\left(\eta^N_{Nt},\xi^N_{Nt}\right)\leq N^{-1}\Delta\left(\eta^N_{0},\xi^N_{0}\right)
+o_N(1)
\ee
where $o_N(1)$ denotes a sequence of random variables converging to $0$ in probability.
\end{definition}
\begin{theorem}(\cite[Theorem 3.2]{bgrs}).
Assume the process is macroscopically stable, enjoys the finite propagation property (Lemma \ref{lemma_finite_prop}).  Assume further that for every $u\in\R$ and every Riemann initial data of the form $\rho_0(.)=R_{\lambda,\rho}(.-u)$, where $R_{\lambda,\rho}$ is defined in \eqref{eq:rie}, there exists an initial sequence $(\eta^N_0)_{N\in\N\setminus\{0\}}$ with profile $\rho_0$ such that the statement of Theorem \ref{th_hydro} holds.
Then Theorem \ref{th_hydro} holds for any initial data $\rho_0\in L^\infty(\R)$ and any initial sequence with profile $\rho_0(.)$.
\end{theorem}
The sequel of this section will be devoted to proving the particular case of Theorem \ref{th_hydro} corresponding to the Riemann problem, that is when $\rho_0=R_{\lambda,\rho}$ defined by \eqref{eq:rie}.
By macroscopic stability, it is actually sufficient to prove this result for a particular sequence of initial 
configurations that we now construct.
 A proof of macroscopic stability for a class of models including ours can be found for instance in \cite[Proposition 2.23]{gs}.\\ \\
\textbf{Equilibria and pseudo-equilibria.} 
Let $(\xi^{\alpha,\rho}_0)_{\rho\in[0,+\infty)}$ denote 
a family of $\mathbf{X}$-valued random configurations such that 
\be\label{ordered_stat}
0\leq\rho\leq\rho'<\rho_c\Rightarrow\xi^{\alpha,\rho}_0\leq\xi^{\alpha,\rho'}_0
\ee
almost surely, and the limits
\be\label{limit_ergo_xi}
\lim_{n\to+\infty}n^{-1}\sum_{x=-n}^0\xi^{\alpha,\rho}_0(x)=
\lim_{n\to+\infty}n^{-1}\sum_{x=0}^n\xi^{\alpha,\rho}_0(x)=\rho
\ee
hold in probability.
Such a family can be constructed in many ways. Let us denote 
by $F_\beta$ the c.d.f. of the probability measure $\theta_\beta$ 
defined in \eqref{eq:theta-lambda}, 
and by $F_{\beta}^{-1}$ the generalized inverse of $F_{ \beta}$.
Let $(V^x)_{x\in\Z}$ be a family of i.i.d. random variables independent of the Harris system, 
such that for every $x\in\Z$, $V^x$ 
is uniformly distributed on $(0,1)$. Then we may set 
\be\label{def_inversion_0}
\xi^{\alpha,\rho}_{0}(x):=F^{-1}_{R^{-1}(\rho)}(V^x)
\ee
Then \eqref{ordered_stat} follows from the fact that 
$(\theta_{R^{-1}(\rho)}:=\theta^\rho)_{\rho\in[0,+\infty)}$ 
is a stochastically nondecreasing family of probability distributions, 
and \eqref{def_inversion_0} yields a monotone coupling of these distributions. 
 Besides, since $\theta^\rho$ has mean $R(R^{-1}(\rho))=\rho$
and the random variables $\xi^{\alpha,\rho}_0$ 
are independent, \eqref{limit_ergo_xi} follows from the law of large numbers. 
Notice that instead of $(\theta^\rho)_{\rho\in[0,+\infty)}$, we could have used 
any other nondecreasing family of distributions parametrized by its mean. 
We could also have used the inversion method to construct deterministic instead 
of i.i.d. configurations in the spirit of \eqref{sample_uniform}--\eqref{inversion}.\\ \\
It may seem more natural to consider a family  $(\xi_0^{\alpha,\rho})$  of stationary processes.
This can be used for instance to infer local equilibrium besides hydrodynamic limit. 
However, the problem of local equilibrium and loss of local equilibrium 
in our setting is deferred to \cite{bmrs4}, where it is investigated in depth.
To obtain stationary processes, one should replace \eqref{def_inversion_0} with
\be\label{def_inversion}
\xi^{\alpha,\rho}_{0}(x):=F^{-1}_{\frac{\overline{R}^{\,\,-1}(\rho)}{\alpha(x)}}(V^x)
\ee
 By Lemma \ref{lemma_inter},  this construction satisfies \eqref{ordered_stat}--\eqref{limit_ergo_xi} 
but is restricted to $\rho\in[0,\rho_c)$. It is not always possible 
to extend this family to a family $(\xi^{\alpha,\rho})_{\rho\in[0,+\infty)}$
satisfying \eqref{ordered_stat}--\eqref{limit_ergo_xi}.
A necessary and sufficient condition for this is
that the invariant measure  $\mu^\alpha_\beta$ defined by \eqref{def_mu_lambda_alpha} satisfies 
\eqref{justify_mean_beta} when $\beta=c$, 
which may not be true  (see Remark \ref{remark_failure}). 
In this case, one may for instance complete \eqref{def_inversion}
by setting, for $\rho>\rho_c$,
\be\label{complete_inversion}
\xi^{\alpha,\rho}_{0}:=\xi^{\alpha,\rho_c}_0+\zeta^{\alpha,\rho-\rho_c}_0
\ee
where $\zeta^{\alpha,r}$ is given by  the r.h.s. of  
\eqref{def_inversion_0}. However, the law of $\xi^{\alpha,\rho}_0$ 
for $\rho>\rho_c$ is no longer invariant for the process with generator 
\eqref{generator}. If 
 $\mu_c^\alpha$ does not satisfy  \eqref{justify_mean_beta},  one may use invariant measures up to 
$\rho_c-\delta$ for any prescribed $\delta>0$,  and complete them above 
this density in a way similar to \eqref{complete_inversion},
setting 
\be\label{complete_inversion_eps}
\xi^{\alpha,\rho}_{0}:=\xi_0^{\alpha,\rho_c-\delta}
+\zeta_0^{\alpha,\rho-\rho_c+\delta}
\ee
As a consequence of \eqref{ordered_stat} and  attractiveness property 
\eqref{attractive_1}, we also have
\be\label{ordered_stat_later}
0\leq\rho\leq\rho'<\rho_c\Rightarrow\xi^{\alpha,\rho}_t\leq\xi^{\alpha,\rho'}_t
\ee
 for $t\ge 0$,  where $(\xi^{\alpha,\rho}_t)_{t\geq 0}$ denotes the process
evolving according to \eqref{generator}  with initial configuration $\xi^{\alpha,\rho}_0$.
 Processes $(\xi^{\alpha,\rho}_.)$ 
that are not stationary (they can {\em never} be if $\rho>\rho_c$) are what we called ``pseudo-equilibria''
at the beginning of this section,
because they are time-invariant on the {\em macroscopic} scale, where they correspond 
to a flat density profile with uniform density $\rho$ at all times.
However, this property does not hold on a smaller scale for supercritical densities,
due to the mass escape at 
slow sites  (see \cite{afgl,fs,bmrs2,bmrs4}).  \\ \\
\textbf{Microscopic Riemann data.}  Using these equilibria and pseudo-equilibria, 
we can construct suitable Riemann states as follows. 
For $s,t\geq 0$ and  $u,v\in\R$,  we set
\be\label{def:xt-yst} 
x_t=\lfloor ut\rfloor,\qquad
y_s^t=\lfloor  ut +vs\rfloor
\ee
  (where $t$ plays the role of a scaling parameter,
 and $s$ is the actual time variable).
For $\lambda,\rho\in\R$, we set  
\begin{eqnarray}\label{def_riemann_inversion}
\eta^{\alpha,\lambda,\rho}_0(x)&:=&\xi_0^{\alpha,\lambda}(x)\indicator{\{x\leq 0\}}+
\xi_0^{\alpha,\rho}(x)\indicator{\{x>0\}}\\
\label{def_riemann_inversion_shift}
\eta_0^{\alpha,\lambda,\rho,t, u}(x)
&:=&\xi_0^{\alpha,\lambda}(x)\indicator{\{x\leq \lfloor ut\rfloor\}}+
\xi_0^{\alpha,\rho}(x)\indicator{\{x>\lfloor ut\rfloor\}}
\end{eqnarray}
 The main step to derive Theorem \ref{th_hydro} for Riemann data 
 is to derive the asymptotic current seen from a moving observer. 
 This is stated in the following proposition,
which is the main result of this section. 
\begin{proposition}\label{prop_hydro_riemann}
For every $\lambda,\rho\in[0,+\infty)$,  $u\in\R$ and $v<1$,  
the following limit holds in probability:
\be\label{current_riemann}
\lim_{t\to+\infty}t^{-1}\Gamma^\alpha_{y^t_.}(
t,\eta_0^{\alpha,\lambda,\rho,t, u})=
{\mathcal G}_{\lambda,\rho}(v)
\ee
 where ${\mathcal G}_{\lambda,\rho}(v)$ was defined in 
\eqref{limiting_current_inf}--\eqref{limiting_current_sup}
of Proposition \ref{proposition_1}. 
\end{proposition}
 The proof of Proposition \ref{prop_hydro_riemann} is performed in Subsection \ref{subsec:proof_riemann},
 using the interface process constructed in Subsection 
\ref{subsec:interface}, and the 
asymptotics of the current for pseudo-equilibria, stated in Proposition \ref{lemma_inv_meas_2} below. The proof of the latter is deferred to Subsection \ref{subsec:current_pseudo}.
 %
%
We now show that Proposition \ref{prop_hydro_riemann} 
indeed implies Riemann hydrodynamics.
\begin{corollary}\label{cor_hydro_riemann}
Theorem \ref{th_hydro} holds for initial data of the form \eqref{eq:rie}.
\end{corollary}
\begin{proof}{Corollary}{cor_hydro_riemann}
 We rely on the notation and statement of Proposition \ref{proposition_1}. 
 It is enough to prove that, for every $v,w\in\R$ such that $v<w$, 
\begin{eqnarray}\nonumber
\lim_{t\to+\infty}t^{-1}\sum_{x=\lfloor ut+vts\rfloor +1}^{\lfloor ut+wts\rfloor} 
\eta^{\lambda,\rho,t,u}_{ts}(x) & = & \int_v^w R_{\lambda,\rho}(x,s)dx\\
 & = & 
 s[{\mathcal G}_{\lambda,\rho}(v)-{\mathcal G}_{\lambda,\rho}(w)] \label{hdl_interval}
\end{eqnarray}
 in probability. Setting $T=ts$ and $U=u/s$, 
by \eqref{difference_currents}, we have
\begin{eqnarray}\nonumber
t^{-1}\sum_{x=\lfloor ut+vts\rfloor +1}^{\lfloor ut+wts\rfloor}\eta^{\lambda,\rho,t,u}_{ts}(x) 
& = & s
T^{-1}\sum_{x=\lfloor UT+vT\rfloor +1}^{\lfloor UT+wT\rfloor}\eta^{\lambda,\rho,T,U}_{T}(x)
\\
& = & 
 s T^{-1}\left(
\Gamma_{Y^T_.}^\alpha(T,\eta^{\lambda,\rho,T,U})-\Gamma_{Z^T_.}^\alpha(T,\eta^{\lambda,\rho,T,U})
\right)\label{setting_wehave}
\end{eqnarray}
where  
 $Y^T_.:=\lfloor UT+v.\rfloor$ and $Z^T_.:=\lfloor VT+w.\rfloor$ .
 Let us assume first that $w<1$. Applying Proposition \ref{prop_hydro_riemann} 
  to $Y^T_.$ and $Z^T_.$,  we obtain 
\be\label{hdl_interval_first}
\lim_{T\to+\infty}T^{-1}\sum_{x=\lfloor UT+vT\rfloor +1}^{\lfloor uT+wT\rfloor}\eta^{\lambda,\rho,T,U}_{T}(x)=
{\mathcal G}_{\lambda,\rho}(v)-{\mathcal G}_{\lambda,\rho}(w)
\ee
 in probability,  which, in view of \eqref{setting_wehave}, 
is equivalent to \eqref{hdl_interval}.  \\ \\
Let us now  prove \eqref{hdl_interval} for $w>1$. Choose $W,V\in\R$ such that $W<1<V<w$. 
By finite propagation property (Lemma \ref{lemma_finite_prop}), on an event $E_T$ 
with probability tending to $1$ as $T\to+\infty$, it holds that 
$\eta^{\alpha,\lambda,\rho,T,U}_{T}(x)=\xi^{\alpha,\rho}_T(x)$ 
for every $x>\lfloor UT+VT\rfloor$. It follows  from \eqref{ergo_density_2}  in Proposition \ref{lemma_inv_meas_2} below that 
\be\label{finiteprop_piece}
\lim_{T\to+\infty}T^{-1}\sum_{x=\lfloor UT+VT\rfloor +1}^{\lfloor UT+wT\rfloor}
\eta^{\alpha,\lambda,\rho,T,U}_T(x)=(w-V)\rho
\ee
in probability.  If $v>1$, we take $V=v$ and we are done. Indeed,
recall from Lemma \ref{lemma_properties_flux} 
that $f$ is $1$-Lipschitz; thus, using \eqref{limiting_current_inf}--\eqref{limiting_current_sup},
 ${\mathcal G}_{\lambda,\rho}(a)=f(\rho)-a\rho$ 
for every $a\geq 1$. Otherwise,  
applying \eqref{hdl_interval_first}  to $W$ yields the limit (still in probability)
\be\label{main_piece}
\lim_{T\to+\infty}T^{-1}\sum_{x=\lfloor UT+vT\rfloor +1}^{\lfloor UT+WT\rfloor}
\eta^{\alpha,\lambda,\rho,T,U}_T (x)=
{\mathcal G}_{\lambda,\rho}(v)-{\mathcal G}_{\lambda,\rho}(W)
\ee
By attractiveness property \eqref{attractive_1}, we have
\[
T^{-1}\sum_{x=\lfloor UT+WT\rfloor +1}^{\lfloor UT+VT\rfloor}\eta^{\alpha,\lambda,\rho,T,U}_T(x)\leq
T^{-1}\sum_{x=\lfloor UT+WT\rfloor +1}^{\lfloor UT+VT\rfloor}\xi^{\alpha,\max(\lambda,\rho)}_T(x)
\]
%
%
Using \eqref{ergo_density_2} again, we have 
\be\label{small_piece}
\quad\lim_{T\to+\infty}{\left\{
\frac{1}{T}\sum_{x=\lfloor UT+WT\rfloor +1}^{\lfloor UT+VT\rfloor}
\eta^{\alpha,\lambda,\rho,T,U}_T(x)-(V-W)\max(\lambda,\rho)\right\}}^+
=0
\ee
Since $W$ and $V$ can be chosen arbitrarily close to $1$, 
and ${\mathcal G}_{\lambda,\rho}$ is continuous, 
\eqref{finiteprop_piece}--\eqref{small_piece} 
imply the limit 
\be\label{all_together}
\qquad\lim_{T\to+\infty}
\frac{1}{T}\sum_{x=\lfloor UT+vT\rfloor +1}^{\lfloor UT+wT\rfloor}\eta^{\alpha,\lambda,\rho,T,U}_T(x)=
{\mathcal G}_{\lambda,\rho}(v)-{\mathcal G}_{\lambda,\rho}(1)+(w-1)\rho
\ee
in probability.  
 So, proceeding as after \eqref{finiteprop_piece},  
the r.h.s. of \eqref{all_together} coincides with that of 
\eqref{hdl_interval_first}. 
\end{proof}
\mbox{}\\ \\
For the proof of Proposition \ref{prop_hydro_riemann}, 
we shall need to know the behaviour of equilibria and pseudo-equilibria processes 
in terms of asymptotic current and hydrodynamic profile, uniformly with respect 
to density. This is stated in the following proposition, which will be proved in Subsection 
\ref{subsec:current_pseudo}.
\begin{proposition}\label{lemma_inv_meas_2}
For $\rho\in[0,+\infty)$, let $(\xi^{\alpha,\rho}_t)_{t\geq 0}$ 
denote the process with initial configuration $\xi^{\alpha,\rho}_0$.
Then, for every $A,B\in\R$ such that $A<B$, every $\varepsilon>0$ 
and every $\rho_0\in[0,+\infty)$, the following limits hold in probability:
\begin{eqnarray}
\lim_{t\to+\infty}
\sup_{\scriptstyle A<a<b<B\atop\scriptstyle b-a>\varepsilon,\,\rho\leq\rho_0}\left|
\frac{1}{(b-a)t}\sum_{x=\lfloor at\rfloor}^{\lfloor bt\rfloor}\xi_t^{\alpha,\rho}(x)
  - \rho\right|&=& 0\label{ergo_density_2}\\
\lim_{t\to+\infty}\sup_{\scriptstyle A<a<B\atop\scriptstyle \rho\leq\rho_0}\left|
\frac{1}{t}\Gamma^\alpha_{\lfloor at\rfloor}(t,\xi_0^{\alpha,\rho})-f(\rho)
\right|&=& 0\label{ergo_flux_2}
\end{eqnarray}
\end{proposition}
%
%
In the above proposition, the  {\em uniformity} with respect to $\rho$ is crucial for our needs. A consequence of this uniformity is that limits similar to \eqref{ergo_density_2}--\eqref{ergo_flux_2} still hold for a {\em random} $\rho$ instead of a deterministic $\rho$, a situation that will arise in Subsection \ref{subsec:interface} and in the sequel. More precisely,  we can state the following corollary to Proposition \ref{lemma_inv_meas_2}:
\begin{corollary}\label{corollary_ab} For every  $A,B\in\R$ such that $A<B$, every $\varepsilon>0$, and every family $(\rho_t)_{t\geq 0}$ of $[0,+\infty)$-valued {\em random variables}, the following limits hold in probability:
\begin{eqnarray}\label{ergo_density_2_random}
\lim_{t\to+\infty}
\sup_{\scriptstyle A<a<b<B\atop\scriptstyle b-a>\varepsilon}\left|
\frac{1}{(b-a)t}\sum_{x=\lfloor at\rfloor}^{\lfloor bt\rfloor}\xi_t^{\alpha,\rho_t}(x)
  - \rho_t\right| & =  & 0 \\
	\lim_{t\to+\infty}\sup_{\scriptstyle A<a<B}\left|
\frac{1}{t}\Gamma^\alpha_{\lfloor at\rfloor}(t,\xi_0^{\alpha,\rho_t})-f\left(\rho_t\right)
\right|&=& 0\label{ergo_flux_2_random}
\end{eqnarray}
\end{corollary}
\begin{remark}\label{remark_ab}
As will be seen in the proof of Proposition \ref{lemma_inv_meas_2}, the reason why the uniformity in Proposition \ref{lemma_inv_meas_2} (and consequently Corollary \ref{corollary_ab}) holds 
is that we are working with a simultaneous {\em monotone} coupling of all the equilibrium (or pseudo-equilibrium) processes. Note that we cannot {\em a priori} relate the {\em distribution} of $\xi_t^{\alpha,\rho_t}$ for a random $\rho$ to the equilibrium distributions of our process, but this will not be 
 needed for our purpose. In other words, if $\rho_t$ is a $[0,\rho_c)$-valued random variable, though one might be tempted to call $\xi_t^{\alpha,\rho_t}$  a ``random equilibrium state'' - in the sense that it is the state at time $t$ of an equilibrium process with a randomly chosen parameter, it is not {\em in general} itself an ``equilibrium state'' - in the sense that its distribution is not necessarily a stationary distribution, though this may happen in particular situations, like for instance if $\rho_t$ is independent of the family of coupled processes $\left(\xi^{\alpha,r}_.\right)_{r\in[0,\rho_c)}$.
\end{remark}
\subsection{The interface process}\label{subsec:interface}
 To construct our interface process,
we shall rely on a property of nearest-neighbour attractive  systems 
(see e.g. \cite[Lemma 4.7]{lig} or \cite[Lemma 6.5]{rez}), namely that
the number of sign changes between the difference of two coupled configurations 
 (through rules \textit{(J1)--(J3)}, that is,  generator
\eqref{coupled_generator}) in   such a system  is a nonincreasing function of time. 
The location of a sign change
can be viewed as an interface,
see also \cite[Lemma 4.3]{bmrs2}  in the context of our model. 
Here we shall explore this property
more precisely by constructing  simultaneous nearest-neighbour dynamics 
for {\em all} interfaces with all equilibria or pseudo-equilibria processes, 
which will define the evolution of a new version of the microscopic density profile, 
whose scaling limit will be investigated. 
 \\ 
The existence and definition of the interface process will be made possible by the following lemma.
In the sequel, without loss of generality, we assume $\lambda\leq\rho$.  For notational simplicity,
we shall henceforth write $\eta_0^{\alpha,\lambda,\rho,t}$ instead of $\eta_0^{\alpha,\lambda,\rho,t,u}$
for the configuration defined by \eqref{def_riemann_inversion_shift}. 
\begin{proposition}
\label{lemma_interface}
There exists a family of processes $({\mathcal X}_s^{\alpha,r,t})_{s\geq 0}$ 
indexed by $r\in[\lambda,\rho]$,
such that 
\be\label{interface_initial}{\mathcal X}_0^{\alpha,r,t}=\lfloor ut\rfloor,\ee
and the following holds:\\ 
(i) For every $r\in[\lambda,\rho]$ and $s\geq 0$, ${\mathcal X}_s^{\alpha,r,t}$ 
is an interface between $\eta_s^{\alpha,\lambda,\rho,t}$ and $\xi_s^{\alpha,r}$ in the sense that 
\be\label{interface_property_2}
\ba{lll}
\dsp\eta_s^{\alpha,\lambda,\rho,t}(y)\leq\xi^{\alpha,r}_s(y)
& \mbox{ for } & \dsp y\leq {\mathcal X}^{\alpha,r,t}_s\\
\dsp \eta_s^{\alpha,\lambda,\rho,t}(y)\geq\xi^{\alpha,r}_s(y) & \mbox{ for } & \dsp {y}> {\mathcal X}^{\alpha,r,t}_s
\ea\ee
(ii) For every $r\in[\lambda,\rho]$, $({\mathcal X}^{\alpha,r,t}_s)_{s\geq 0}$ is 
a piecewise constant c\`adl\`ag $\Z$-valued process
with nearest-neighbour jumps. \\ 
(iii) For every $r,r'\in[\lambda,\rho]$ and every $s\geq 0$, it holds that
\be\label{order_interface_2}
r\leq r'\Rightarrow {\mathcal X}^{\alpha,r,t}_s\leq {\mathcal X}^{\alpha,r',t}_s
\ee
(iv) For every $r\in[\lambda,\rho]$ and  $t>0$, there exist  Poisson processes 
${\mathcal N}_.^{\pm, r,  t}$ with intensity $1$ such that, for all $s\geq 0$,
\be\label{poisson_bounds}-{\mathcal N}_s^{-, r,  t}\leq {\mathcal X}^{\alpha,r,t}_s-{\mathcal X}^{\alpha,r,t}_0\leq{\mathcal N}_s^{+, r,  t}\ee
\end{proposition}
Proposition \ref{lemma_interface} will be proved at the end of this subsection.
Observe that, since ${\mathcal X}^{\alpha,r,t}_.$ is $\Z$-valued and monotone 
with respect to $r$, as a function of $r$ 
(for fixed $s$ and $t$), it is a step function. We may define its generalized inverses:
\begin{eqnarray}\label{inverse_interface_minus}
{\mathcal R}^{-,\alpha,x,t}_s & := & \sup\left\{
r\in[\lambda,\rho]:\,{\mathcal X}^{\alpha,r,t}_s<x
\right\}\\
\label{inverse_interface_plus}
{\mathcal R}^{+,\alpha,x,t}_s & := & \inf\left\{
r\in[\lambda,\rho]:\,{\mathcal X}^{\alpha,r,t}_s>x
\right\}\end{eqnarray}
 for $x\in\R$. 
Since ${\mathcal X}^{\alpha,r,t}_.$ takes integer values, ${\mathcal R}^{+,\alpha,.,t}_s$
and ${\mathcal R}^{-,\alpha,.,t}_s$ have the same constant value on $(x,x+1)$ for every $x\in\Z$. 
Both $({\mathcal X}^{\alpha,r,t}_s)_{s\geq 0, r\geq 0}$ 
and $({\mathcal R}^{\pm,\alpha,x,t}_s)_{s\geq 0,x\in\Z}$
 will be called  the {\em interface process}. The latter 
is an approximation of the (monotone) hydrodynamic profile, 
while the former is an approximation of its inverse, 
which gives the positions of the different density levels 
$r$ of the profile. We shall see below that, after rescaling, 
they do converge to the profile and inverse profile. 
It follows from \eqref{interface_property_2}, \eqref{order_interface_2}  and
\eqref{inverse_interface_minus}--\eqref{inverse_interface_plus} that, for any $x,y\in\Z$ such that
$x<y$, and any $s\geq 0$,
\be\label{property_inverse_interface}
\xi^{\alpha,r^-}_s( z)\leq\eta^{\alpha,\lambda,\rho,t}_s( z)
\leq\xi^{\alpha,r^+}_s( z),\quad\mbox{for all }
z\in(x,y)\cap\Z
\ee
provided $r^-$ and $r^+$ satisfy 
\be\label{densities_inverse_interface}
0<r^-<{\mathcal R}^{-,\alpha,x,t}_s, \quad
r^+ >{\mathcal R}^{+,\alpha,y,t}_s,\quad
\ee
In particular, in a region where  
 ${\mathcal R}^{+,\alpha,x,t}_s$ and ${\mathcal R}^{-,\alpha,x,t}_s$  do not vary too much 
with $x$  and remain close to each other, the process is close to 
$\xi^{\alpha,\rho}_s$ for some {\em random} value of $\rho$. 
We may view this as a coupling formulation 
of the local equilibrium property, as this means that the configuration is locally close to that of an equilibrium (or pseudo-equilibrium) configuration with random density parameter. Recall however from Remark \ref{remark_ab} that, due to the randomness of this parameter, this coupling information does not a priori translate into an information on the local distribution of $\eta^{\alpha,\lambda,\rho}_s$ - like in particular being close to a stationary distribution, and that such information is not necessary to our purpose. 
On the other hand,
whenever  ${\mathcal R}^{+,\alpha,x,t}-{\mathcal R}^{-,\alpha,x,t}$ 
is of order one, this can be interpreted as the presence of a shock at microscopic location $x$.\\ \\
In order to prove Proposition \ref{prop_hydro_riemann}, 
we shall have to study limits of the time-rescaled processes
\begin{eqnarray}\label{rescaled_interface_1}
x^{\alpha,t}(r,s) & := & t^{-1}{\mathcal X}^{\alpha,r,t}_{ts}\\
\rho^{\pm,\alpha,t}(y,s) & : = & {\mathcal R}^{\pm,\alpha,ty,t}_{ts}\label{rescaled_interface_2}
\end{eqnarray}
defined for $r\in[\lambda,\rho]$, $s\geq 0$ and $y\in\R$. 
%
%
For every $x\in\Z$,
the restrictions of $\rho^{+,\alpha,t}(.,s)$
and $\rho^{-,\alpha,t}(.,s)$ to $(x/t,(x+1)/t)$ have the same constant value. We denote these common restrictions by $\rho^{\alpha,t}(.,s)$:
\be\label{common_value}
\quad\rho^{\alpha,t}(y,s):=\rho^{+,\alpha,t}(y,s)=\rho^{-,\alpha,t}(y,s),\quad\forall x\in\Z,\,\forall y\in\left(\frac{x}{t},\frac{x+1}{t}\right)
\ee 
%
%
Note that, as functions of $r$ and $y$, 
$x^{\alpha,t}(.,s)$ and $\rho^{\pm,\alpha,t}(.,s)$ are 
generalized inverses of each other. 
 We next define a convenient topology to study limits of these rescaled interfaces. \\ \\
 For $a,b\in\R$ such that $a<b$, let ${\mathcal F}_{\lambda,\rho}^{a,b}$ 
denote the set of nondecreasing functions $\psi$ on $\R$ such that
$\psi(x)=\lambda$ for $x<a$ and $\psi(x)=\rho$ for $x>b$. An element 
$\psi$ of ${\mathcal F}_{\lambda,\rho}^{a,b}$ can be identified 
with its derivative, that is a measure on $\R$ supported on $[a,b]$ 
with mass $\rho-\lambda$. The generalized inverse $\psi^{-1}$ of $\psi$ lies in the set 
${\mathcal E}_{\lambda,\rho}^{a,b}$ of nondecreasing functions on 
$[\lambda,\rho]$ with value $a$ at $\lambda$ and $b$ at $\rho$. 
An element of ${\mathcal E}_{\lambda,\rho}^{a,b}$ is identified 
with its derivative, that is a measure on $[\lambda,\rho]$ with 
mass $b-a$. In the sequel, we identify measures on $[a,b]$ with 
measures on $\R$ supported on $[a,b]$. We denote by 
${\mathcal M}_{a,b,m}$ the set of measures on $[a,b]$ with mass 
no greater than $m$, and we equip this set with the topology of 
weak convergence, for which it is compact. By Helly's theorem, 
the notion of convergence induced on either set 
${\mathcal F}_{\lambda,\rho}^{a,b}$ or ${\mathcal E}_{\lambda,\rho}^{a,b}$ 
is that of pointwise convergence at every continuity point of the limiting function.
With these topologies, 
\be\label{involution}
\mbox{the involution }\psi\mapsto\psi^{-1}\mbox{ between }
{\mathcal F}_{\lambda,\rho}^{a,b}\mbox{ and }
{\mathcal E}_{\lambda,\rho}^{a,b}\mbox{ is bicontinuous.}
\ee
For $T>0$, we let $\widetilde{\mathcal M}_{a,b,m,T}$ 
denote the set of continuous functions from $[0,T]$ to 
 ${\mathcal M}_{a,b,m}$ equipped with the topology of uniform convergence. 
In the next subsection, in parallel to Proposition \ref{prop_hydro_riemann}, 
we shall prove its following counterpart in terms of the interface process.
\begin{proposition}\label{prop_hydro_interface}
For every  $T>0$ and $V>1$,\\ 
(i)  the processes $(\rho^{\pm,\alpha,t}(.,s))_{s\geq 0}$ 
converge in probability  in $\widetilde{\mathcal M}_{u-VT,u+VT,\rho-\lambda,T}$   
as $t\to+\infty$ to the deterministic process $(\rho(.,s))_{s\geq 0}$, 
where $\rho(x,s)={\mathcal R}_{\lambda,\rho}(x -u,s)$ is the solution 
given by \eqref{def_riemann_sol} of the Riemann problem
\eqref{conservation_law}  with initial datum  \eqref{eq:rie} 
 centered at $u$; \\ 
(ii)  the process $(x^{\alpha,t}(.,s))_{s\geq 0}$ 
converges in probability  in $\widetilde{\mathcal M}_{\lambda,\rho,2VT,T}$   
as $t\to+\infty$ to the deterministic trajectory $(x(r,s))_{s\geq 0}$  such that, 
for every $s\geq 0$, $x(.,s)$ is the generalized inverse of $\rho(.,s)$. 
\end{proposition}
\begin{remark}\label{remark_charac}
In fact (as alluded to in the introduction), $x^\alpha(r,.)$ can be 
interpreted as a generalized characteristic for the conservation law 
\eqref{conservation_law}. This will be subtantially developed in a forthcoming paper.
\end{remark}
 To prepare the proof of Proposition \ref{prop_hydro_interface}, 
 we first need a tightness result with respect to the topology introduced above. \\
\begin{proposition}\label{prop_limit_rho}
 For every $T>0$ and $V>1$,\\
 (i) the family of processes 
$(x^{\alpha,t}(.,s))_{s\geq 0}$ is tight in
$\widetilde{\mathcal M}_{\lambda,\rho,2VT,T}$;\\
 (ii) the  family  of processes 
$(\rho^{+,\alpha, t}(.,s),\rho^{-,\alpha, t }(.,s))_{s\geq 0}$ 
is tight in \\
$\widetilde{\mathcal M}_{u-VT,u+VT,\rho-\lambda,T}$, 
and any subsequential weak limit of this  sequence 
is a random pair $(\rho^{+,\alpha}(.,s),\rho^{-,\alpha}(.,s))_{s\geq 0}$ 
of elements of 
$\widetilde{\mathcal M}_{u-VT,u+VT,\rho-\lambda,T}$.  Besides, almost 
surely with respect to the law of this pair, it holds that 
for all $s\geq 0$, $\rho^{+,\alpha}(.,s)=\rho^{-,\alpha}(.,s)=:\rho^{\alpha}(.,s)$ 
a.e.  on $[0,+\infty)$. 
\end{proposition}
\begin{proof}{Proposition}{prop_limit_rho}
 Let ${\mathcal Y}^{\alpha,t}_s:=\lfloor ut\rfloor+{\mathcal N}_s^{+,r,t}$ and
${\mathcal Z}^{\alpha,t}_s:=\lfloor ut\rfloor-{\mathcal N}_s^{-,r,t}$. 
By \textit{(i)} and \textit{(iv)} of Proposition \ref{lemma_interface}, 
we have $\rho^{\pm,\alpha,t}(y,s)=\lambda$ for  $y<{\mathcal Z}^{\alpha,t}_s$ and
$\rho^{\pm,\alpha,t}(y,s)=\rho$ for $y>{\mathcal Y}^{\alpha,t}_s$. 
Besides, by the law of large numbers for Poisson processes, 
$t^{-1}{\mathcal Y}^{\alpha,t}_{st}$ and 
$t^{-1}{\mathcal Z}^{\alpha,t}_{st}$ converge in probability respectively 
to $u+s$ and $u-s$. Hence, with probability tending to $1$ as $t\to+\infty$, 
for every $s\in[0,T]$, $x^{\alpha,t}(.,s)$ lies in 
${\mathcal E}_{\lambda,\rho}^{u-Vs,u+Vs}$ 
(thus in ${\mathcal M}_{\lambda,\rho,2Vs}$) and $\rho^{\pm,\alpha,t}(.,s)$ lies in
${\mathcal E}_{\lambda,\rho}^{u-Vs,u+Vs}$ (thus in ${\mathcal M}_{u-Vs,u+Vs,\rho-\lambda}$).   \\ \\
 Remark that \textit{(ii)} follows from \textit{(i)} and \eqref{involution}. 
Now we show  point \textit{(i)}.  
To this end, it is enough to show that for every continuous test function 
$\varphi$ on $[\lambda,\rho]$, the family of processes 
$(x^{\alpha,t}_.(\varphi))_{t\in[0,T]}$ defined by the Stieltjes integral
\[
x^{\alpha,t}_s(\varphi):=\int_\lambda^\rho \varphi(r)x^{\alpha,t}(dr,s)
\]
 is tight. 
Equivalently, we shall show it for piecewise constant functions $\varphi$  of the form
\[
\varphi(r)=\sum_{k=0}^{n-1}\varphi_k 1_{(r_k,r_{k+1}]}(r)
\]
where $n\geq 1$ and $\lambda=r_0<r_1<\cdots r_n=\rho$, since such functions uniformly approximate 
continuous functions on $[\lambda,\rho]$. In this case, we have
\be\label{stieltjes}
x^{\alpha,t}_s(\varphi)=\sum_{k=0}^{n-1}\varphi_k\left[
x^{\alpha,t}(r_{k+1},s)-x^{\alpha,t}(r_k,s)
\right]
\ee
The problem thus boils down to proving that the family $(x^{\alpha,t}(r,.))_{t\geq 0}$ of processes 
 is tight in $C^0([0,T];\R)$. 
By  Proposition \ref{lemma_interface},\textit{(iv)} and Markov property,
for any  $0\leq s$, $({\mathcal X}^{\alpha,r,t}_{s+\tau}-{\mathcal X}^{\alpha,r,t}_s)_{\tau\geq 0}$ 
is stochastically dominated by a rate $1$ Poisson process. 
Hence, for every $\varepsilon>0$ and inverse integer $\delta>0$, 
\begin{eqnarray}
\Prob\left(
\sup_{0\leq s<s'\leq T}|x^{\alpha,t}_{s'}(\varphi)-x^{\alpha,t}_s(\varphi)|>\varepsilon
\right)   & \leq &  \frac{1}{\delta}\Prob\left(
\frac{1}{\delta Tt}{\mathcal P}(\delta Tt)>\frac{\varepsilon}{\delta T}
\right)\nonumber\\
& \leq &  \frac{1}{\delta}e^{-t
I_{\delta T}(\varepsilon)
}\label{poisson_ld}
\end{eqnarray}
where 
\[ 
I_{\delta T}(\varepsilon)
:=
\varepsilon
\ln\frac{\varepsilon}{\delta T}-\varepsilon+\delta T
\]
Inequality \eqref{poisson_ld} follows from cutting the interval $[0,T]$ 
into intervals of length $\delta T$ and using Poisson large deviation bounds. 
Choosing $\delta>0$, we obtain $I_\delta(\varepsilon)>0$, 
hence the modulus of continuity of $x^{\alpha,t}_.$ vanishes in probability as $t\to+\infty$.
\end{proof}
\mbox{}\\ \\
 We conclude this subsection, as announced, with the proof of Proposition \ref{lemma_interface}.  \\
\begin{proof}{Proposition}{lemma_interface}
At time $s=0$, Properties \textit{(i)} and \textit{(iii)} 
hold thanks to \eqref{def_riemann_inversion} and \eqref{ordered_stat}.
 Note that if  $\eta_0^{\alpha,\lambda,\rho,t}(y)=\xi^{\alpha,r}_0(y)$
 for $y\in [a,b]\cap\Z,a,b\in\Z$, then we can take ${\mathcal X}^{\alpha,r,t}_0=y$
 for any $y\in [a-1,b]\cap\Z$. \\ 
We define the evolution of the interface position ${\mathcal X}^{\alpha,r,t}_.$ as follows.
Assume ${\mathcal X}^{\alpha,t,r}_{s-}=x$.
This position is only possibly modified at time $s$ if a clock 
from our Harris construction rings at time $s$ and position $x$ or $x+1$: 
that is, if $\omega(\{(s,w,u,z\})=1$, for some $w\in\{x,x+1\}$, $u\in[0,1]$ and  $z\in\{-1,1\}$, 
where $\omega$ is the Poisson measure defined by \eqref{def_omega}. We then update 
the interface position or not according to the following rules:\\ \\
{\em Case 1.} 
Assume \eqref{case_1}--\eqref{case_12} below hold: 
\be \label{case_1} \ba{l}
\dsp\alpha(w)g\left(\eta^\alpha_{s-}(w)\right)<u\leq\alpha(w)g\left(\xi^{\alpha,r}_{s-}(w)\right),\\ \\
\dsp\eta^{\alpha}_{s-}(w+z)-\xi^{\alpha,r}_{s-}(w+z)=0 
\ea\ee
\be\label{case_12} w=x,\,z=1,
\ee
 By rule \textit{(J2),} this means a  potential jump  from $x$ to $x+1$. Then
we set 
\be\label{go_right}
{\mathcal X}^{\alpha,r,t}_s=x+1
\ee
{\em Case 2.} Assume \eqref{case_2}--\eqref{case_22} below hold: 
\be\label{case_2}
\ba{l}
\alpha(w)g\left(\xi^{\alpha,r}_{s-}(w)\right)<u
\leq\alpha(w)g\left(\eta^\alpha_{s-}(w)\right), \\ \\
\eta^{\alpha}_{s-}(w+z)-\xi^{\alpha,r}_{s-}(w+z)=0,
\ea\ee
\be
\label{case_22} \qquad  w=x+1,\,z=-1
\ee
 By rule \textit{(J2)},  this means a potential jump  from $x+1$ to $x$. Then
we set 
\be\label{go_left}
{\mathcal X}^{\alpha,r,t}_s=x-1
\ee
{\em Case 3.} If neither  \eqref{case_1}--\eqref{case_12} 
nor \eqref{case_2}--\eqref{case_22}  holds, we set
\be\label{stay_there}
{\mathcal X}^{\alpha,r,t}_s=x
\ee
The above rules satisfy property \textit{(ii)}. 
We now prove that they do satisfy the other properties as well.\\ \\
{\em Proof of (i).}\\ 
{\em Case 1.} The first condition in \eqref{case_1} implies that 
a $\xi^{\alpha,r}_{s-}$ particle jumps from  $w=x$ to $w+z=x+1$  at time $s$,
 without being accompanied by a $\eta^{\alpha}_{s-}$ particle.
Since $g$ is nondecreasing, this condition also implies 
$\eta^\alpha_{s-}(x)-\xi^{\alpha,r}_{s-}(x)<0$. 
 After this jump, we have $\eta^\alpha_{s}(x)-\xi^{\alpha,r}_{s}(x)\leq 0$, 
 and due to the second condition in \eqref{case_1}, we also have 
 $\eta^\alpha_{s}(x+1)-\xi^{\alpha,r}_{s}(x+1)<0$. 
 \\ \\
{\em Case 2.} The first condition in \eqref{case_2} implies that 
an $\eta^{\alpha}_{s-}$ particle jumps from  $w=x+1$ to $w+z=x$  at time $s$, 
without being accompanied by a $\xi^{\alpha,r}_{s-}$ particle.
Since $g$ is nondecreasing, this condition also implies 
$\eta^\alpha_{s-}(x+1)-\xi^{\alpha,r}_{s-}(x+1)> 0$.  After this jump,
 we have $\eta^\alpha_{s}(x+1)-\xi^{\alpha,r}_{s}(x+1)\geq 0$, and 
 due to the second condition in \eqref{case_2}, we also have 
 $\eta^\alpha_{s}(x)-\xi^{\alpha,r}_{s}(x)>0$. 
 \\ \\
 In both cases, 
 for any $y\in\Z\setminus\{w,w+z\}$, the sign of 
 $\eta^\alpha_{s}(y)-\xi^{\alpha,r}_{s}(y)$ is the same as that of 
 $\eta^\alpha_{s-}(y)-\xi^{\alpha,r}_{s-}(y)$. Therefore, 
 property \textit{(i)} holds at time $s$, respectively with \eqref{go_right}
 in Case 1, and with \eqref{go_left} in Case 2.  \\ \\
We now consider all possibilities in Case 3.\\ \\
{\em Case (a).}  $w=x$ and the first condition in \eqref{case_1} 
does not hold, or $w=x+1$ and the first condition in  \eqref{case_2}  does not hold.
Since property \textit{(i)} at time $s-$ implies 
$\eta^\alpha_{s-}(x)-\xi^{\alpha,r}_{s-}(x)\leq 0$  and 
$\eta^\alpha_{s-}(x+1)-\xi^{\alpha,r}_{s-}(x+1)\leq 0$, by rules \textit{(J1)--(J3)}, 
 either no 
 particle 
jumps from $w$ to $w+z$, or both an $\eta^\alpha_{s-}$ particle and a 
$\xi^{\alpha,r}_{s-}$ particle  do.
\\ \\
{\em Case (b).}  $w=x$, $z=1$, the first condition in \eqref{case_1} 
holds but not the second one. As in case 1 above, the former condition 
implies $\eta^\alpha_{s-}(x)-\xi^{\alpha,r}_{s-}(x)< 0$. By property \textit{(i)} 
at time $s-$  and the latter condition, 
$\eta^\alpha_{s-}(x+1)-\xi^{\alpha,r}_{s-}(x+1) > 0$. 
Thus, $\eta^\alpha_{s}(x)-\xi^{\alpha,r}_{s}(x)\leq 0$ and 
$\eta^\alpha_{s}(x+1)-\xi^{\alpha,r}_{s}(x+1)\geq 0$.
\\ \\
{\em Case (c).} $w=x+1$, $z=-1$, the first condition in \eqref{case_2} 
holds but not the second one. As in case 2 above, the former condition implies 
$\eta^\alpha_{s-}(x+1)-\xi^{\alpha,r}_{s-}(x+1)>0$. 
By property \textit{(i)} at time $s-$, we must have 
$\eta^\alpha_{s-}(x)-\xi^{\alpha,r}_{s-}(x)\leq 0$. 
Thus, $\eta^\alpha_{s}(x)-\xi^{\alpha,r}_{s}(x)\leq 0$ and 
$\eta^\alpha_{s}(x+1)-\xi^{\alpha,r}_{s}(x+1)\geq 0$.
\\ \\
{\em Case (d).} $w=x$, $z=-1$ and the first condition in \eqref{case_1} holds, 
so that $\eta^\alpha_{s-}(x)<\xi^{\alpha,r}_{s-}(x)$ and a $\xi^{\alpha,r}_{s-}$ 
particle alone jumps from $x$ to $x-1$ at time $s$. 
By property \textit{(i)} at time $s-$, $\eta^\alpha_{s-}(x-1)-\xi^\alpha_{s-}(x-1)\leq 0$. 
At time $s$, we have $\eta^\alpha_{s}(x)-\xi^{\alpha,r}_{s}(x)\leq 0$  and 
$\eta^\alpha_{s}(x-1)-\xi^{\alpha,r}_{s}(x-1)< 0$. 
\\ \\
{\em Case (e).} $w=x+1$, $z=1$ and the first condition in \eqref{case_2} holds, 
so that $\eta^\alpha_{s-}(x+1)-\xi^{\alpha,r}_{s-}(x+1)>0$ and  an $\eta^\alpha_{s-}$ 
particle alone jumps from $x+1$ to $x+2$ at time $s$. 
By property \textit{(i)} at time $s-$, $\eta^\alpha_{s-}(x+2)-\xi^\alpha_{s-}(x+2)\geq 0$. 
At time $s$, we have $\eta^\alpha_{s}(x+1)-\xi^{\alpha,r}_{s}(x+1)\geq 0$ and 
$\eta^\alpha_{s}(x+2)-\xi^{\alpha,r}_{s}(x+2)> 0$. 
 \\ \\
 In Case 3\textit{(a)}, respectively in all the
other subcases of Case 3, for any $y\in\Z$;
respectively  
 for any $y\in\Z\setminus\{w,w+z\}$, the sign of 
 $\eta^\alpha_{s}(y)-\xi^{\alpha,r}_{s}(y)$ is the same as that of 
 $\eta^\alpha_{s-}(y)-\xi^{\alpha,r}_{s-}(y)$. Therefore, 
 property \textit{(i)} holds at time $s$ with \eqref{stay_there}.  \\ \\
{\em Proof of (iii).} We have to prove that, whenever $r<r'$ and 
\be\label{same_interface}
{\mathcal X}^{\alpha,r,t}_{s-}={\mathcal X}^{\alpha,r',t}_{s-}=x,
\ee
then ${\mathcal X}^{\alpha,r,t}_{s}\leq {\mathcal X}^{\alpha,r',t}_{s}$. 
We must thus check that if ${\mathcal X}^{\alpha,r,t}_.$ jumps to the right, 
then ${\mathcal X}^{\alpha,r',t}_.$ does the same, and if 
${\mathcal X}^{\alpha,r',t}_.$ jumps to the left, then ${\mathcal X}^{\alpha,r,t}_.$ 
does the same.  \\ \\
{\em Case 1'.} If ${\mathcal X}^{\alpha,r,t}_.$ jumps  to the right, we are in Case 1 above. 
Since (by \eqref{ordered_stat_later}) $\xi^{\alpha,r'}_{s-}\geq\xi^{\alpha,r}_{s-}$, 
the first condition in \eqref{case_1} is also satisfied for density $r'$, 
and the second condition in \eqref{case_1} for $r$ implies 
$\eta^{\alpha}_{s-}(x+1)\leq\xi^{\alpha,r'}_{s-}(x+1)$. But \eqref{same_interface} 
and property \textit{(i)} at time $s-$ imply 
$\eta^\alpha_{s-}(x+1)\geq\xi^{\alpha,r'}_{s-}(x+1)$, hence 
$\eta^\alpha_{s-}(x+1)=\xi^{\alpha,r'}_{s-}(x+1)$, thus the second condition 
in \eqref{case_1} is also satisfied for $r'$, so that ${\mathcal X}^{\alpha,r',t}$ 
jumps  to the right.  \\ \\
{\em Case 2'}. If ${\mathcal X}^{\alpha,r',t}_.$ jumps  to the left, we are in Case 2 above. 
Since (by \eqref{ordered_stat_later}) $\xi^{\alpha,r'}_{s-}\geq\xi^{\alpha,r}_{s-}$, 
the first condition in \eqref{case_2} is also satisfied for density $r$, 
and the second condition in \eqref{case_2} for  $r'$  implies 
$\eta^{\alpha}_{s-} (x)  \geq\xi^{\alpha,r}_{s-}(x)$. 
But \eqref{same_interface} and
 property \textit{(i)} at time $s-$ imply $\eta^\alpha_{s-}(x)\leq\xi^{\alpha,r}_{s-} (x) $, 
 hence $\eta^\alpha_{s-}(x)=\xi^{\alpha,r}_{s-} (x) $, thus the second condition 
 in \eqref{case_2} is also satisfied for $r$, so that ${\mathcal X}^{\alpha,r,t}$ jumps 
 to the  left.   \\ \\
{\em Proof of (iv).} The way we defined the evolution of ${\mathcal X}_.^{\alpha,r, t}$ 
shows that \eqref{poisson_bounds} holds if we define ${\mathcal N}_.^{\pm, r,  t}$ as follows. 
At time $s=0$, we set
\[
{\mathcal N}_0^{+, r,  t}={\mathcal N}_0^{-, r,  t}=\lfloor ut\rfloor
\]
Then, for the evolution of these processes, 
at time $s$, 
if ${\mathcal N}_{s-}^{+, r, t}=x$,  we set
${\mathcal N}_s^{+, r, t}=x+1$
if and only if $\omega\left(\{(s,x)\}\times[0,1]\times\{1\}\right)=1$, 
otherwise ${\mathcal N}_s^{+, r,  t}=x$.
Similarly, if $\omega\left(\{(s,x)\}\times[0,1]\times\{-1\}\right)=1$, we set 
${\mathcal N}_s^{-, r, t} =x-1 $, otherwise 
${\mathcal N}_s^{-, r, t}=x$. The processes 
$({\mathcal N}_s^{\pm, r, ,t})_{s\geq 0}$ 
defined in this way are Poisson processes with intensity $1$.
\end{proof}
\subsection{Riemann hydrodynamics: proofs of Propositions \ref{prop_hydro_riemann} 
and \ref{prop_hydro_interface}}\label{subsec:proof_riemann}
Recall that  we  assume $\lambda\leq\rho$. The natural analogues of the
propositions below for 
$\lambda\geq\rho$ are proven entirely in the same way.  Recall definitions \eqref{def:xt-yst}
of $x_t$ and $y^t_s$. 
We start proving the easier part of Proposition \ref{prop_hydro_riemann}, that is the upper bound.
\begin{proposition}\label{prop_hydro_riemann_upper}
For every $\lambda\leq\rho\in[0,+\infty)$,  $t>0$ and u,$v\in\R$, 
the following limit holds in probability.
\be\label{current_riemann_upper}
\lim_{t\to+\infty}\left[
t^{-1}\Gamma^\alpha_{y^t_.}(t,\eta^{\alpha,\lambda,\rho,t}_0)-{\mathcal G}_{\lambda,\rho}(v)
\right]^+=0
\ee
\end{proposition}
\begin{proof}{Proposition}{prop_hydro_riemann_upper}
By Lemma \ref{lemma_current} and \eqref{difference_currents} we have,  for $r\in[\lambda,\rho]$, 
\begin{eqnarray}\label{compare_currents_interface}
t^{-1}\Gamma^\alpha_{y^t_.}(t,\eta^{\alpha,\lambda,\rho,t}_0) 
& \leq & t^{-1}\Gamma^\alpha_{y^t_.}(t,\xi^{\alpha,r}_0)\\
t^{-1}\Gamma^\alpha_{y^t_.}(t,\xi^{\alpha,r}_0)
 & = &  t^{-1}\Gamma^\alpha_{x_t-1}(t,\xi^{\alpha,r}_0)-t^{-1}\sum_{x=x_t}^{y_t^t}\xi_t^{\alpha,r}(x)
\label{decomp_currents_interface}
\end{eqnarray}
 By Proposition \ref{current_uniform}  below, the first term on the r.h.s. of
 \eqref{decomp_currents_interface} converges a.s. to the mean current $f(r)$. 
On the other hand, by   \eqref{ergo_density_2} of  
Proposition \ref{lemma_inv_meas_2}, the second term
converges in distribution to $-vr$. This yields \eqref{current_riemann_upper}, recalling definition \eqref{limiting_current_inf} of $\mathcal G_{\lambda,\rho}$.
\end{proof}
\mbox{}\\ \\
We now proceed to the proof of the lower bound in Proposition \ref{prop_hydro_riemann},  which will be carried out in parallel to that of Proposition \ref{prop_hydro_interface}.
\begin{proposition}\label{prop_hydro_riemann_lower}
 For every $\lambda\leq\rho\in[0,+\infty)$,  $t>0$,  $u\in\R$ and $v<1$,  the following limit holds in probability
\be\label{current_riemann_lower}
\lim_{t\to+\infty}\left[
t^{-1}\Gamma^\alpha_{y^t_.}(t,\eta^{\alpha,\lambda,\rho,t}_0)-
{\mathcal G}_{\lambda,\rho}(v)
\right]^-=0
\ee
\end{proposition}
Let us summarize the general idea to prove Propositions \ref{prop_hydro_interface}
and \ref{prop_hydro_riemann_lower}. In order to estimate the current across our ``observer'' 
$y^t_.$ travelling at speed $v$, we consider our process $\eta^{\alpha,\lambda,\rho}_.$ at times $ts$ for $s\in[0,1]$ (here, $t\to+\infty$
plays the role of a scaling parameter, and $s\in[0,1]$ that of a macroscopic time variable). We use the interface process to compare our process around 
the observer to 
$\xi^{\alpha,\rho}_{ts}$ with a priori {\em random} density 
$\rho=\rho^{\pm,\alpha,t}(t^{-1}y^t_{ts},s)$, where $t^{-1}y^t_{ts}$ is  the macroscopic  
location of our observer at macroscopic time $s$ (that is microscopic time $ts$).
Roughly speaking, if we are lucky enough that the traveller never sits on a shock, 
the interface processes $\rho^{\pm,\alpha,t}(.,ts)$ 
will be close to each other and not significantly vary in the neighbourhood of $t^{-1}y^t_{ts}$,
which means that the process in the neighbourhood of  $y^t_{ts}$  is close to $\xi^{\alpha,\rho}_{ts}$ for both values
$\rho=\rho^{\pm,\alpha,t}(t^{-1}y^t_{ts},s)$. 
Using the uniform estimates \eqref{ergo_density_2}--\eqref{ergo_flux_2}  from
Proposition \ref{lemma_inv_meas_2},  we can show 
that the instantaneous current across $y^t_.$ is close to 
$f(r)-vr$ for $r=\rho^{\pm,\alpha,t}(t^{-1}y^t_{ts},s)$, which has the desired 
lower bound ${\mathcal G}_{\lambda,\rho}(v)$. 
Since this holds for every $s\in[0,1]$, by integrating the instantaneous current over $s$, we obtain the same bound for the total current between $s=0$ and $s=1$, which is the statement of Proposition \ref{prop_hydro_riemann_lower}.
Besides,  the minimizer in 
${\mathcal G}_{\lambda,\rho}(v)$ is precisely unique outside a shock, 
and equal to the value $\rho(u+vs,s)$ of the entropy solution. Thus
$\rho^{\pm,\alpha,t}(t^{-1}y^t_{ts}.,s)$ must be close to 
$\rho(u+vs,s)$.\\ \\
However, we cannot {\em a priori} discard that $v$ is precisely the speed of a shock, 
because we cannot specify $v$ to avoid shocks {\em before} knowing where they are, 
which would require knowing that the limit we are trying to prove does hold. 
This is why we shall first replace  $\Gamma^\alpha_{y^t_.}$  by $\Gamma^{\alpha,L}_{y^t_.}$ 
defined below in \eqref{def_average_gamma}, that is a local spatial average of 
$\Gamma^\alpha_{y^t_.}$ over the observer's position. Indeed, shocks are isolated, 
so the above argument should be true almost everywhere along this spatial average. 
We point out that our interface-based definition of local equilibrium 
(or pseudo-equilibrium) remains valid at supercritical densities, which is not 
the case for the usual weak-convergence based approach, due to lack of invariant measures.\\ \\ 
Let $L\in\N\setminus\{0\}$.  We define spatial averages of currents as follows. 
\begin{eqnarray}\label{def_average_gamma}
\Gamma^{\alpha,L}_{\lfloor y^t_.\rfloor}(t,\eta^{\alpha,\lambda,\rho,t}_0)
&:=& L^{-1}\sum_{i=0}^{L-1}
\Gamma^{\alpha}_{\lfloor y^t_.+i\rfloor}(t,\eta^{\alpha,\lambda,\rho,t}_0)\\
\label{def_average_gamma_2}
\Gamma^{\alpha,L}_{\lfloor y^t_.\rfloor}(s,s',\eta^{\alpha,\lambda,\rho,t}_0)
&:=& L^{-1}\sum_{i=0}^{L-1}
\Gamma^{\alpha}_{\lfloor y^t_.+i\rfloor}(s,s',\eta^{\alpha,\lambda,\rho,t}_0)
\end{eqnarray}
By \eqref{difference_currents},
\begin{eqnarray}
t^{-1}\Gamma^\alpha_{\lfloor y^t_.\rfloor}(t,\eta^{\alpha,\lambda,\rho,t}_0) & \geq &
t^{-1}\Gamma^{\alpha,L}_{\lfloor y^t_.\rfloor}(t,\eta^{\alpha,\lambda,\rho,t}_0)
\nonumber\\
&-&
t^{-1}L^{-1}\sum_{i=1}^{L-1}\sum_{x=\lfloor ut\rfloor+1}^{\lfloor ut\rfloor+i}\eta^{\alpha,\lambda,\rho,t}_0(x)
\label{observe_eq_loc}
\end{eqnarray}
By next lemma, 
the last term on the r.h.s. of \eqref{observe_eq_loc} can be neglected. 
\begin{lemma}\label{lemma_neglect}
For $L=\lfloor\varepsilon t\rfloor$, it holds that
\be\label{eq_neglect}
\lim_{\varepsilon\to 0}\lim_{t\to+\infty} 
 \Exp  \left\{
t^{-1}L^{-1}
\sum_{i=1}^{L-1}\sum_{x=\lfloor ut\rfloor+1}^{\lfloor ut\rfloor+i} 
\eta^{\alpha,\lambda,\rho,t}_0(x)
\right\}=0
\ee
\end{lemma}
\begin{proof}{Lemma}{lemma_neglect}
By   definition \eqref{def_riemann_inversion_shift} of $\eta^{\alpha,\lambda,\rho,t}_0$  
we have 
$\eta^{\alpha,\lambda,\rho,t}_0\leq 
\xi^{\alpha,\rho}_0$. The result then follows from \eqref{limit_ergo_xi}.
\end{proof}
\mbox{}\\ \\
 From now on,  we shall always assume $L=\lfloor\varepsilon t\rfloor$
 for $\varepsilon$ positive but arbitrarily small.  
Propositions \ref{prop_hydro_interface} and \ref{prop_hydro_riemann_lower}
will essentially be consequences of  Proposition \ref{prop_hydro_riemann_lower_average}
below. 
\begin{proposition}\label{prop_hydro_riemann_lower_average}
For $a<b$, $m>0$, $\delta>0$ and $r(.,.)\in\widetilde{\mathcal M}_{a,b,m, T}$, define
\be\label{limcurrent}
\gamma^{\varepsilon,\delta}[r(.,.)]:=\frac{\delta}{2}\sum_{k=0}^{\lfloor 2/\delta\rfloor}
\gamma^{\varepsilon,\delta}_k[r(.,.)]
\ee
where 
\begin{eqnarray}\nonumber\label{limcurrent_k}
\gamma^{\varepsilon,\delta}_k[r(.,.)]&:=&\frac{1}{\varepsilon}\int_0^\varepsilon \left\{
f\left[r\left(u+\frac{vk\delta}2+z-\delta,\frac{k\delta}2\right)\right]\right.\\
&&\qquad\qquad\qquad\left. -vr\left(u+\frac{v k\delta}2+z-\delta,\frac{k\delta}2\right)
\right\}dz
\end{eqnarray}
Then,  for every $\varepsilon>0$, $h>0$,  
$\lambda,\rho\in[0,+\infty)$, $t>0$ and u,$v\in\R$, 
\be\label{current_riemann_lower_average}
\lim_{\delta\to 0}\limsup_{t\to+\infty}\Prob\left(
t^{-1}\Gamma^{\alpha,\lfloor\varepsilon t\rfloor}_{y^t_.}(t,\eta^{\alpha,\lambda,\rho,t}_0)<
\gamma^{\varepsilon,\delta}[\rho^{-,\alpha,t}(.,.)]-h
\right)=0
\ee
\end{proposition}
 Before proving Proposition \ref{prop_hydro_riemann_lower_average}, 
we prove that it implies Proposition \ref{prop_hydro_interface} 
and Proposition \ref{prop_hydro_riemann_lower}.\\ 
\begin{proof}{Propositions}{prop_hydro_interface} 
{\em and \ref{prop_hydro_riemann_lower}}.
In the following, using Proposition \ref{prop_limit_rho}, we consider 
any sequence of values of $t$ tending to $+\infty$ along which the process 
$(\rho^{+,\alpha, t}(.,s),\rho^{-,\alpha, t }(.,s))_{s\geq 0}$ 
converges in law to some $(\rho^\alpha(.,s))_{s\geq 0}$. It will be implicit in the notation
that $t\to+\infty$ will mean a limit along this subsequence.  \\ 
Propositions \ref{prop_hydro_riemann_upper}, \ref{prop_hydro_riemann_lower_average}, 
Lemma \ref{lemma_neglect} and \eqref{observe_eq_loc} imply that, for any $h>0$,
\be\label{near_min}
\lim_{\varepsilon\to 0}\lim_{\delta\to 0}\limsup_{t\to+\infty}\Prob\left(
\gamma^{\varepsilon,\delta}[\rho^{-,\alpha,t}(.,.)]>{\mathcal G}_{\lambda,\rho}(v)+h
\right)=0
\ee
Since $\gamma^{\varepsilon,\delta}$ is a continuous functional, 
and $\rho^\alpha(.,s)\in\widetilde{\mathcal M}_{a,b,m}$ 
for some $a,b,m$, $\gamma^{\varepsilon,\delta}(\rho^{-,\alpha,t})$ 
converges in law as $t\to+\infty$ to 
$\gamma^{\varepsilon,\delta}(\rho^{\alpha})$ 
and $\gamma^{\varepsilon,\delta}(\rho^{\alpha})$ converges a.s. as $\delta\to 0$ 
to $\gamma^{\varepsilon}(\rho^{\alpha})$,  where $\gamma^\varepsilon$ 
is defined on $\widetilde{\mathcal M}_{a,b,m}$ by 
\be \label{limcurrent_2}\quad
%
\gamma^{\varepsilon}(r)  
:= \frac{1}{\varepsilon}
\int_0^1\int_0^\varepsilon
 \left\{
f[r(u+vs+z-,s)] - v r(u+vs+z- ,s) 
\right\}dzds 
\ee
Finally, since $\rho^\alpha\in\widetilde{\mathcal M}_{a,b,m}$, 
$\gamma^\varepsilon(\rho^\alpha)$ converges a.s. (with respect to the 
distribution of the random function $\rho^\alpha$) 
as $\varepsilon\to 0$ to $\gamma(\rho^\alpha)$, where $\gamma(.)$ is defined by
\be\label{limcurrent_3}
\gamma(r):=\int_0^1 \left\{
f[r(u+vs+,s)]-vr(u+vs+,s)
\right\}ds
\ee
It follows from \eqref{near_min} and definition \eqref{limiting_current_inf} 
of $\mathcal G_{\lambda,\rho}(v)$ that 
\be\label{good_gamma}
\gamma(\rho^\alpha)={\mathcal G}_{\lambda,\rho}(v) \mbox{ a.s.}
\ee
Equality \eqref{good_gamma}, together with \eqref{limiting_current_inf}, 
\eqref{density_inf} and \eqref{def_riemann_sol}, implies \textit{(i)} 
of Proposition \ref{prop_hydro_interface}. The latter combined with 
\eqref{involution} implies \textit{(ii)} of Proposition \ref{prop_hydro_interface}. Finally,
\eqref{good_gamma}  and \eqref{current_riemann_lower_average}  establish Proposition \ref{prop_hydro_riemann_lower}.
\end{proof}
\mbox{}\\
\begin{proof}{Proposition}{prop_hydro_riemann_lower_average}
We shall compute 
$\Gamma^{\alpha,L}_{\lfloor y^t_.\rfloor}(t,\eta^{\alpha,\lambda,\rho,t}_0)$ by decomposing the time interval 
$[0,t]$ into subintervals of length $t\delta/2$ denoted by $tI_k:=[tk\delta/2,t(k+1)\delta/2)=[ts_k,ts_{k+1})$ for 
$k=0,\ldots,K-1$ for 
\be\label{number_intervals}
K:=\lfloor 2/\delta\rfloor
\ee
and a last interval $tI_K:=[tK\delta/2,t]=[ts_K,ts_{K+1}]$, where $ts_{K+1}=t$. We thus write
\be\label{decomp_current}
\Gamma^{\alpha,L}_{\lfloor y^t_.\rfloor}(t,\eta^{\alpha,\lambda,\rho,t}_0)=\sum_{k=0}^{K}
\Gamma^{\alpha,L}_{\lfloor y^t_.\rfloor}(ts_k,ts_{k+1},\eta^{\alpha,\lambda,\rho,t}_{ts_k})
\ee
In the sequel, for notational simplicity, we shall write (for $i=0,\ldots,L-1$)
\be\label{for_simplicity}
\ba{lll}
\rho^{-}_{k,i} & : = & \rho^{-,\alpha,t}\left(t^{-1}y^t_{ts_k}+t^{-1}i-\delta,s_k\right)-\delta\\
\rho^{+}_{k,i} & : = & \rho^{+,\alpha,t}\left(t^{-1}y^t_{ts_k}+t^{-1}i+\delta,s_k\right)+\delta\\
\xi_{k,i}^\pm & := & \xi_{ts_k}^{\rho^\pm_{k,i}}\\
\xi_{k,k+1,i}^\pm  & := & \xi_{ts_{k+1}}^{\rho^\pm_{k,i}} \\
 \eta_k & := & \eta^{\alpha,\lambda,\rho,t}_{ts_k} 
\ea\ee
These processes represent a discretization in our analysis.  
By \eqref{property_inverse_interface}--\eqref{densities_inverse_interface}
and \eqref{rescaled_interface_1}--\eqref{rescaled_interface_2}, we have 
\be\label{sandwich_interface}
\xi_{k,i}^-(x)\leq
\eta^{\alpha,\lambda,\rho,t}_{ts_k}(x)\leq
\xi_{k,i}^+(x)
\ee
for every $i\in\{0,\ldots,L-1\}$ and $x\in\mathcal U$, where
\be\label{name_interval_0}
\mathcal U:=(y^t_{ts_k}-t\delta+i,y^t_{ts_k}+t\delta+i)
\ee
By Lemma \ref{lemma_current}, there is an event $E_t^{\varepsilon,\delta}$ 
with probability tending to $1$ as $t\to+\infty$, on which the following holds 
for every $k=0,\ldots,K$ and $i=0,\ldots,\lfloor\varepsilon t\rfloor$, $V\in(1,2)$ 
and $v\in\R$ such that $V+v<2$:
\begin{eqnarray}\label{compare_current_eq}
& & \Gamma^{\alpha}_{\lfloor y^t_.+i\rfloor}(ts_k,ts_{k+1},\eta_k)-
\Gamma^{\alpha}_{\lfloor y^t_.+i\rfloor}(ts_k,ts_{k+1},\xi_{k,i}^-)\\
\nonumber& \geq & -0\vee\max\left\{
F_{y^t_{ts_k}+i}(\eta_k,x)-
F_{y^t_{ts_k}+i}(\xi_{k,i}^-,x):\,
x\in\mathcal V
\right\}\\
\nonumber & \geq & -\sum_{x=y^t_{ts_k}+i-V\delta/2}^{y^t_{ts_k}+i+(V+v)\delta/2}\left[
\xi_{k,i}^+(x)-\xi_{k,i}^-(x)
\right]
\end{eqnarray}
where
\be\label{name_interval}
\mathcal V:=[y^t_{ts_k}+i-Vt\delta/2,y^t_{ts_k}+i+(V+v)t\delta/2]
\ee
Notice that, thanks to the condition $V+v<2$, the interval $\mathcal V$ 
defined by \eqref{name_interval} is indeed contained in the interval 
$\mathcal U$ defined by \eqref{name_interval_0}.
Thus, on $E_t^{\varepsilon,\delta}$, 
\begin{eqnarray}\label{thus}
t^{-1}\Gamma^{\alpha,\lfloor\varepsilon t\rfloor}_{\lfloor y^t_.\rfloor}(ts_k,ts_{k+1},\eta_k) & \geq &
 \Gamma^{1,\varepsilon,\delta}_k(t)-
\Gamma^{2,\varepsilon,\delta}_{k}(t)
\end{eqnarray}
where
\begin{eqnarray*}
\Gamma^{1,\varepsilon,\delta}_k(t)
& := &
\frac{1}{\lfloor\varepsilon t\rfloor}\sum_{i=0}^{\lfloor\varepsilon t\rfloor-1}
t^{-1}\Gamma^{\alpha}_{\lfloor y^t_.+i\rfloor}(ts_k,ts_{k+1},\xi_{k,i}^-)\\
\Gamma^{2,\varepsilon,\delta}_{k}(t)
& := &
 \frac{1}{\lfloor\varepsilon t\rfloor}\sum_{i=0}^{\lfloor\varepsilon t\rfloor-1}
t^{-1}\sum_{x=y^t_{ts_k}+i- t\delta/2}^{y^t_{ts_k}+i+(V+v) t\delta/2}\left[
\xi_{k,i}^+(x)-\xi_{k,i}^-(x)
\right]
\end{eqnarray*}
There, $\Gamma^{1,\varepsilon,\delta}_k(t)$ is the essential item, 
that is the current through the expectedly close process $\xi^{\alpha,\rho_{k,i}^-}_.$ 
whose limiting value is given by 
Corollary \ref{corollary_ab} (mind here that $\rho_{k,i}^-$ is a {\em random variable}).
On the other hand,
$\Gamma^{2,\varepsilon,\delta}_{k}(t)$ is the error, which is controlled 
(see \eqref{second_quantity} below)
by the microscopic jump of the interface process. The latter is negligible 
in the absence of a macroscopic shock, so it will be negligible after spatial averaging. 
Below we replace these terms by their ``main'' values, which are functions of the 
interface processes $\rho^{\pm,\alpha,t}$ rather than the particle processes. \\ \\ 
For $\Gamma^{1,\varepsilon,\delta}_k(t)$ in \eqref{thus} we write, 
for each $i=0,\ldots,\lfloor\varepsilon t\rfloor-1$,  using \eqref{difference_currents}, 
\begin{eqnarray}\nonumber
t^{-1}\Gamma^\alpha_{y^t_.+i}(ts_k,ts_{k+1},\xi^-_{k,i}) & = &
t^{-1}\Gamma^\alpha_{ y_{ts_k}^t+i}(ts_k,ts_{k+1},\xi^-_{k,i})
\\&&\qquad -t^{-1}
\sum_{x=y_{ts_k}^t+i+1}^{y_{ts_{k+1}}^t+i}\xi^-_{k,k+1}(x)
 \label{thus_one}
\end{eqnarray}
Hence,  by  Corollary \ref{corollary_ab},
%
\begin{eqnarray}\label{hence}
\Gamma^{1,\varepsilon,\delta}_k(t) & \geq & 
\widetilde{\Gamma}^{1,\varepsilon,\delta}_k(t)-e^{\varepsilon,\delta}_k(t)
\end{eqnarray}
where $e^{\varepsilon,\delta}_k(t)\to 0$ in probability as $t\to+\infty$, and 
\begin{eqnarray}
\widetilde{\Gamma}^{1,\varepsilon,\delta}_k(t) &
:= & \frac{1}{\lfloor\varepsilon t\rfloor-1}\sum_{i=0}^{\lfloor\varepsilon t\rfloor}\left[
f(\rho_{k,i}^-)-v\rho_{k,i}^-
\right]\nonumber\\
& = & \frac{1}{\varepsilon}\int_0^\varepsilon \left\{
f[\rho^{-,\alpha,t}(t^{-1}y^t_{ts_k}+z-\delta,s_k)]\right.\nonumber\\
&  &  \phantom{\frac{1}{\varepsilon}\int_0^\varepsilon\left\{\right.}\left.
-v\rho^{-,\alpha,t}(t^{-1}y^t_{ts_k}+z-\delta,s_k)
\right\}dz\nonumber\\
 & = & \gamma^{\varepsilon,\delta}_k(\rho^{-,\alpha,t})\label{limit_thus_one}
\end{eqnarray}
To obtain the first line of \eqref{limit_thus_one}, we applied  \eqref{ergo_flux_2_random} to the first line on the r.h.s. of \eqref{thus_one}, and  \eqref{ergo_density_2_random} to the second line on the r.h.s. of \eqref{thus_one}. To obtain the second line of \eqref{limit_thus_one},
 observe that by \eqref{for_simplicity}, the integrand on the second line 
 of \eqref{limit_thus_one} is  constant on intervals of
length $1/t$, and the first line is the corresponding Riemann sum.  
 For $\Gamma^{2,\varepsilon,\delta}_k(t)$ in \eqref{thus}, 
by Proposition \ref{lemma_inv_meas_2}, for each $k=0,\ldots,K$,
\be\label{replace_gamma2}
\lim_{t\to +\infty}\left[\Gamma^{2,\varepsilon,\delta}_{k}(t)-
\widetilde{\Gamma}^{2,\varepsilon,\delta}_{k} (t)\right]=0\quad
\mbox{in probability}
\ee
where 
\begin{eqnarray}
\widetilde{\Gamma}^{2,\varepsilon,\delta}_{k}(t) 
& := &\frac{\delta (2V+v)}{2\lfloor\varepsilon t\rfloor}
\sum_{i=0}^{\lfloor\varepsilon t\rfloor-1}
\left[
\rho_{k,i}^+-\rho_{k,i}^-
\right]\nonumber\\
&= & \frac{\delta (2V+v)}{2\varepsilon}\int_0^\varepsilon\left[
\rho^{+,\alpha,t}\left(t^{-1}y^t_{ts_k}+z+\delta,s_k\right)\right.\nonumber\\
&& \phantom{\frac{\delta (2V+v)}{2\varepsilon}\int_0^\varepsilon}\left.
-\rho^{-,\alpha,t}\left(t^{-1}y^t_{ts_k}+z-\delta,s_k\right)
\right]dz\nonumber\\
& + & \delta^2 (2V+v)\nonumber\\
& = & \frac{\delta (2V+v)}{2\varepsilon}\left\{
\int_{\varepsilon-\delta}^{\varepsilon+\delta}\rho^{\alpha,t}\left(u+v s_k+z,s_k\right)dz\right.\nonumber\\
&&\phantom{\frac{\delta (2V+v)}{2\varepsilon}}\left.-\int_{-\delta}^\delta\rho^{\alpha,t}\left(u+v s_k+z,s_k\right)dz
\right\}\nonumber\\
& + & \delta^2 (2V+v)\nonumber\\
& \leq & \frac{2\rho\delta^2(2V+v)}{\varepsilon}
+ \delta^2 (2V+v)\label{second_quantity}
\end{eqnarray}
 is the ``main part'' of ${\Gamma}^{2,\varepsilon,\delta}_{k}(t)$. 
Note that in the third equality, we replaced the functions $\rho^{\pm,\alpha,t}$ 
with $\rho^{\alpha,t}$, as all these functions coincide a.e. in space (see \eqref{common_value}). 
Hence
\be\label{third_quantity}
\sum_{k=0}^K\widetilde{\Gamma}^{2,\varepsilon,\delta}_{k}(t)\leq
\frac{2\rho\delta (2V+v)}{\varepsilon}+\delta(2V+v) 
\ee
By \eqref{decomp_current} and \eqref{thus}, 
\begin{eqnarray*}
 t^{-1}  \Gamma^{\alpha,L}_{\lfloor y^t_.\rfloor}(t,\eta^{\alpha,\lambda,\rho,t}_0) & = & \sum_{k=0}^{K}
 t^{-1} \Gamma^{\alpha,L}_{\lfloor y^t_.\rfloor}(ts_k,ts_{k+1},\eta^{\alpha,\lambda,\rho,t}_{ts_k})\\
& \geq & 
\sum_{k=0}^K 
{\Gamma}^{1,\varepsilon,\delta}_k(t)
-
\sum_{k=0}^K {\Gamma}^{2,\varepsilon,\delta}_k(t)
\end{eqnarray*}
Hence, 
the conclusion follows from 
\eqref{hence}, \eqref{limit_thus_one},  \eqref{replace_gamma2} and \eqref{third_quantity}.
\end{proof}
 \subsection{Proof of Proposition \ref{lemma_inv_meas_2}}\label{subsec:current_pseudo}
This subsection is devoted to the proof of Proposition \ref{lemma_inv_meas_2}. 
The proof will be carried out in two steps. The main step will be to prove
a nonuniform result for the asymptotic current at a given density, that we now state.
\begin{proposition}\label{current_uniform}
Let $\eta_0\in\mathbf{X}$ be an initial (deterministic or random) configuration satisfying,
 for some $\rho\in[0,+\infty)$, 
\be\label{assumption_uniform}
\lim_{n\to+\infty}\frac{1}{n}\sum_{x=0}^n\eta_0(x)=\lim_{n\to+\infty}\frac{1}{n}\sum_{x=-n}^0\eta_0(x)=\rho
\ee
 in probability. 
Then, 
for any $x\in\R$ and $t>0$, the following limit holds in probability 
with respect to the law of the 
process:
\begin{eqnarray}
\lim_{N\to+\infty}N^{-1}\Gamma^\alpha_{\lfloor Nx\rfloor}(Nt,\eta_0) & = & tf(\rho)\label{limit_current}
\end{eqnarray}
\end{proposition}
Before proving Proposition \ref{current_uniform}, we use it to derive  Proposition
\ref{lemma_inv_meas_2}.\\
\begin{proof}{Proposition}{lemma_inv_meas_2}
\mbox{}\\ \\
{\em Proof of \eqref{ergo_density_2}.} 
 Let $a,b\in\R$ such that $a<b$, and  $\rho\geq 0$. 
Since  by \eqref{difference_currents} we have 
\[
\Gamma^\alpha_{\lfloor at\rfloor-1}(t,\xi^{\alpha,\rho}_0)
-\Gamma^\alpha_{\lfloor bt\rfloor}(t,\xi^{\alpha,\rho}_0)
=\sum_{x=\lfloor at\rfloor}^{\lfloor bt\rfloor}\xi^{\alpha,\rho}_t(x)
-\sum_{x=\lfloor at\rfloor}^{\lfloor bt\rfloor}\xi^{\alpha,\rho}_0(x),
\]
applying \eqref{limit_ergo_xi} and Proposition \ref{current_uniform} gives 
\be\label{simple_conv}
\lim_{t\to+\infty}\left|\frac{1}{(b-a)t}
\sum_{\lfloor at\rfloor}^{\lfloor bt\rfloor}\xi_t^{\alpha,\rho}(x)  - \rho\right|=0
\ee
in probability. 
The stronger uniform result is obtained using attractiveness 
and a discretization of  densities and positions.
Indeed, for $n\in\N\setminus\{0\}$ and $k=0,\ldots,n$, 
let 
\be\label{valdisc}
r_k^n:=\frac{k}{n}\rho_0, \qquad x_k^n:= A+\frac{k}{n}(B-A)
\ee 
\begin{eqnarray*}
S(t) & := &
\sup_{\scriptstyle A<a<b<B\atop\scriptstyle b-a>\varepsilon,\,\rho\leq\rho_0}
\left|
\frac{1}{(b-a)t}
\sum_{\lfloor at\rfloor}^{\lfloor bt\rfloor}\xi_t^{\alpha,\rho}(x)  - \rho\right|
\\
S^n(t) & := &
\max_{\scriptstyle k,l,m=0,\ldots,n\atop\scriptstyle m-l>n\varepsilon}
\left|\frac{1}{(x^n_m-x^n_l)t}
\sum_{x=\lfloor x^n_l t\rfloor}^{\lfloor x^n_m t\rfloor}\xi_t^{\alpha,r^n_k}(x)  - r^n_k\right|
\end{eqnarray*}
By \eqref{simple_conv}, for each $n\in\N\setminus\{0\}$, $S^n(t)$ 
converges to $0$ in probability as $t\to+\infty$. 
On the other hand, $S(t)\leq S^n(t)+E^n(t)$, where $E^n(t)$ 
is an upper bound for the discretization error. Using \eqref{ordered_stat_later}, 
we can take as upper bound
\begin{eqnarray}\nonumber
E^n(t) & := & \max_{\scriptstyle k,l,m=0,\ldots,n\atop\scriptstyle m-l>n\varepsilon}
\frac{1}{(x^n_m-x^n_l)t}\sum_{x=\lfloor x^n_l t\rfloor}^{\lfloor x^n_m t\rfloor}\left[
\xi^{\alpha,r^n_{k+1}}_t(x)-\xi^{\alpha,r^n_{k}}_t(x)
\right]
+\frac{\rho_0}{n}\\
& + & \frac{2}{n\varepsilon}\max_{k=0,\ldots,n}\max_{l=0,\ldots,n-1}
\frac{1}{(x^n_{l+1}-x^n_l)t}
\sum_{x=\lfloor x^n_l t\rfloor}^{\lfloor x^n_{l+1} t\rfloor}\xi^{\alpha,r^n_k}_t(x)
\label{discretization_bound}
\end{eqnarray}
It follows from \eqref{simple_conv} that 
\[
\lim_{t\to+\infty}E^n(t)  \le  2\frac{\rho_0}{n}+2\frac{\rho_0}{n\varepsilon}
\]
in probability. Since $n$ can be taken arbitrarily large, 
 \eqref{ergo_density_2} follows.  \\ \\
{\em Proof of \eqref{ergo_flux_2}.} 
We use the same discretization  \eqref{valdisc}  as previously, now setting
\begin{eqnarray*}
T(t) & := & \sup_{\scriptstyle A<a<B\atop\scriptstyle \rho\leq\rho_0}\left|
\frac{1}{t}\Gamma^\alpha_{\lfloor Na\rfloor}(t,\xi_0^{\alpha,\rho})-f(\rho)
\right|\\ 
T^n(t) & := &\max_{k,l=0,\ldots,n}\left|
\frac{1}{t}\Gamma^\alpha_{\lfloor tx^n_l\rfloor}(t,\xi_0^{\alpha,r^n_k})-f(r^n_k)
\right|
\end{eqnarray*}
By Proposition \ref{current_uniform},  for each $n\in\N\setminus\{0\}$, $T^n(t)$ 
converges to $0$ in probability as $t\to+\infty$. 
We again write $T(t)\leq T^n(t)+F^n(t)$, but now the discretization bound $F^n(t)$ 
is controlled using Lemma \ref{lemma_current} and \eqref{difference_currents}, which yields
\begin{eqnarray}
\nonumber F^n(t) & := & 
\max_{k=0,\ldots,n,\,l=0,\ldots,n-1}t^{-1}\sum_{x=\lfloor tx^n_l\rfloor}^{\lfloor tx^n_{l+1}\rfloor}\xi^{\alpha,r^n_k}_t(x)\\
\label{discretization_bound_2}
& + & \max_{k=0,\ldots,n-1}t^{-1}\sum_{x=\lfloor (A-V)t\rfloor}^{\lfloor (B+V)t\rfloor}\left[
\xi^{\alpha,r^n_{k+1}}_t(x)-\xi_{ 0}^{\alpha,r^n_{k}}(x)
\right]
\end{eqnarray}
It follows from \eqref{simple_conv} that 
\[
\lim_{t\to+\infty}F^n(t) \le \frac{\rho_0}{n}+(B-A+2V)\frac{\rho_0}{n}
\]
in probability, so we may conclude as previously.
\end{proof}
\mbox{}\\ \\
The sequel of this subsection is devoted to the proof of Proposition \ref{current_uniform}.
In the case $\rho<\rho_c$, Proposition \ref{current_uniform} was proven in \cite[Lemma 4.10]{bmrs2}
for subcritical equilibria, that is $\eta_0=\xi_0^{\alpha,\rho}$ defined by \eqref{def_inversion}.
Indeed, \cite[Lemma 4.10]{bmrs2} was valid only for $x\le 0$
because in \cite{bmrs2} only the second limit in \eqref{assumption_uniform}  was assumed;
since \eqref{assumption_uniform}  gives two limits, the proof of
\cite[Lemma 4.10]{bmrs2} is also valid for $x>0$. 
The following lemma shows that it implies the same for any $\eta_0$ satisfying 
\eqref{assumption_uniform} with $\rho<\rho_c$. 
\begin{lemma}\label{lemma_observe}
Assume \eqref{limit_current} holds for {\em some} $\eta_0\in{\mathbf{X}}$ satisfying \eqref{assumption_uniform}.
Then it holds for {\em any} $\eta'_0\in{\mathbf{X}}$ satisfying \eqref{assumption_uniform}.
\end{lemma}
\begin{proof}{Lemma}{lemma_observe}
Since both $\eta_0$ and $\eta'_0$ satisfy \eqref{assumption_uniform}, 
Lemma \ref{lemma_current} implies,  for any $x\in\R, t>0$,  the limit in probability
\be\label{implies_aslimit}
\lim_{N\to+\infty}
\left|
N^{-1}\Gamma^\alpha_{\lfloor Nx\rfloor}\left(Nt,\eta_0\right)
-N^{-1}\Gamma^\alpha_{\lfloor Nx\rfloor}\left(Nt,\eta'_0\right)
\right|
=0
\ee 
\end{proof}
\mbox{}\\
To complete  the proof of Proposition  \ref{current_uniform}, 
we now treat the case $\rho\geq\rho_c$ as follows.  Recall \eqref{extension_flux}.  
\begin{proposition}\label{current_super}
Under assumption \eqref{assumption_uniform} for $\rho\geq\rho_c$, 
\begin{eqnarray}
\lim_{N\to+\infty}N^{-1}\Gamma^\alpha_{\lfloor Nx\rfloor}(Nt,\eta_0) 
& = &  t f(\rho) =  t(p-q)c\label{limit_current_super}
\end{eqnarray}
\end{proposition}
For the proof of Proposition \ref{current_super}, 
we need to define the following quantities, 
for $\kappa\in{\bf A}$  an arbitrary environment and $\varepsilon>0$. 
\begin{eqnarray}\label{def:ell}
A_\varepsilon(\gendis)&:=&\sup\{x\le 0:\gendis(x) \leq c+\varepsilon\}\in \Z^-\cup\{-\infty\}\,,\\
a_\varepsilon(\gendis)&:=&\inf\{x\ge 0:\gendis(x) \leq c+\varepsilon\}\in\overline{\N}\,.\label{def:err}
\end{eqnarray} 
with the usual conventions $\inf\emptyset=+\infty$ and $\sup\emptyset=-\infty$.
 It follows from the above definitions that 
\begin{eqnarray}\label{limit_a_epsilon}
&&\lim_{\varepsilon\to 0}A_\varepsilon(\gendis)=-\infty,\quad
\lim_{\varepsilon\to 0}a_\varepsilon(\gendis)=+\infty \\
\label{finite_a_min}
&&\liminf_{x\to-\infty}\gendis(x)=c\Rightarrow\forall \varepsilon>0,\,A_\varepsilon(\gendis)>-\infty\\
\label{finite_a_plus}
&&\liminf_{x\to+\infty}\gendis(x)=c\Rightarrow\forall \varepsilon>0,\,a_\varepsilon(\gendis)<+\infty
\end{eqnarray}
Coming back to the setting of Proposition \ref{current_super}, 
\eqref{finite_a_min}--\eqref{finite_a_plus} and 
\eqref{not_too_sparse} iùmply that $A_\varepsilon(\alpha)$ 
and $a_\varepsilon(\alpha)$ are finite. Besides, 
a consequence of  Assumption \ref{assumption_dense}  is the following: 
\begin{lemma}\label{lemma_ergodic} 
For every $\varepsilon>0$,
\be\label{limit_ergodic}
\lim_{n\to\pm\infty}n^{-1}a_\varepsilon(\tau_n\alpha)=0,\quad
\lim_{n\to\pm\infty}n^{-1}A_\varepsilon(\tau_n\alpha)=0
\ee
\end{lemma}
\begin{proof}{Lemma}{lemma_ergodic}
Consider for instance the first limit.  
For any $n\in\Z$, there exists a unique $k(n)\in\Z$ such that
$x_{k(n)-1}<n\leq x_{k(n)}$, where $(x_k)_{k\in\Z}$ is the sequence 
in Assumption \ref{assumption_dense}.  Since $\lim_{n\to+\infty}k(n)=+\infty$, 
by  \eqref{assumption_afgl}, 
for $n$ large enough, we have  $\alpha\left(x_{k(n)}\right)<c+\varepsilon$.  Hence 
\[n^{-1}a_\varepsilon\left(\tau_n\alpha\right)\leq n^{-1}\left[ x_{k(n)}-n\right]
\leq n^{-1}\left[ x_{k(n)}-x_{k(n)-1}\right]\]
which vanishes as $n\to+\infty$ by  \eqref{assumption_afgl_0}. 
\end{proof}\mbox{}\\ 
\begin{proof}{Proposition}{current_super}
To derive \eqref{limit_current_super}, 
we establish first an upper bound, then a lower bound, that is
\begin{eqnarray}
\lim_{N\to\infty}\Exp\left[
N^{-1}\Gamma^\alpha_{\lfloor Nx\rfloor}\left(Nt,\eta_0\right)-t(p-q)c
\right]^+ & = & 0\label{current_upperbound} \\
\label{current_lowerbound}
\liminf_{N\to+\infty}\Exp\left[
N^{-1}\Gamma^\alpha_{\lfloor Nx\rfloor}\left(Nt,\eta_0\right)
\right] & \geq & t(p-q)c
\end{eqnarray}
{\em Step one: 
upper bound  \eqref{current_upperbound}.} 
Let $y_N:=\lfloor Nx\rfloor + A_\varepsilon(\tau_{\lfloor Nx\rfloor}\alpha)$. 
By \eqref{difference_currents},
\be\label{compare_current_A}
N^{-1}\Gamma_{\lfloor Nx\rfloor }^\alpha(Nt,\eta_0)\leq 
N^{-1}\Gamma_{y_N}^\alpha(Nt,\eta_0)+
N^{-1}\sum_{x=y_N+1}^{\lfloor Nx\rfloor}\eta_0(x)
\ee
By Corollary \ref{corollary_consequence},  
\[
\Gamma^\alpha_{y_N}\left(Nt,\eta_0\right)\leq
\Gamma^\alpha_{y_N}\left(Nt,\eta^{*,y_N}\right)
\]
Applying Proposition \ref{current_source} to the right-hand side, and using the fact that 
$\alpha(y_N)\leq c+\varepsilon$ (by definition \eqref{def:ell} of  $A_\varepsilon(.))$,  we obtain
\be\label{current_upperbound_weobtain}
\limsup_{N\to+\infty}\Exp\left\{
\left[
N^{-1}\Gamma^\alpha_{y_N}\left(Nt,\eta_0\right)-t(p-q)c
\right]^+
\right\}\leq\varepsilon t
\ee
By assumption \eqref{assumption_uniform} and Lemma \ref{lemma_ergodic}, 
the second term on the r.h.s. of \eqref{compare_current_A} vanishes as $N\to+\infty$. 
The upper bound \eqref{current_upperbound} then follows from 
\eqref{current_upperbound_weobtain} and \eqref{compare_current_A} by letting $\varepsilon\to 0$.\\ \\
{\em Step two: lower bound \eqref{current_lowerbound}.}
Let $\delta>0$,  and set $\eta_0^{\alpha,\delta}:=\xi_0^{\alpha,\rho}$ defined by \eqref{complete_inversion_eps}.
We are going to prove that
\be\label{super_current_2}
\liminf_{\delta\to 0}\liminf_{N\to\infty}\Exp\left\{
N^{-1}\Gamma^\alpha_{\lfloor Nx\rfloor}\left(Nt,\eta_0^{\alpha,\delta}\right)
\right\}
\geq t(p-q)c
\ee
Indeed,  since both $\eta_0$ and $\eta_0^{\alpha,\delta}$ satisfy \eqref{assumption_uniform}, 
Lemma \ref{lemma_current} implies the limit in probability
\be\label{implies_aslimit}
\lim_{N\to+\infty}
\left|
N^{-1}\Gamma^\alpha_{\lfloor Nx\rfloor}\left(Nt,\eta_0^{\alpha,\delta}\right)
-N^{-1}\Gamma^\alpha_{\lfloor Nx\rfloor}\left(Nt,\eta_0\right)
\right|
=0
\ee
Thus  \eqref{super_current_2} implies \eqref{current_lowerbound}.      \\ \\
 \textit{Proof of \eqref{super_current_2}.} 
We use \eqref{difference_currents} to write 
\be\label{for98}
\qquad  N^{-1}\Gamma^\alpha_{\lfloor Nx\rfloor}\left(Nt,\eta_0^{\alpha,\delta}\right)-
N^{-1}\Gamma^\alpha_{ z_N}\left(Nt,\eta_0^{\alpha,\delta}\right)
\geq-N^{-1}\sum_{x=1+\lfloor Nx\rfloor}^{  z_N }\eta_0^{\alpha,\delta}(x)\ee
\be\label{def_y_N}
\mbox{where}\qquad z_N:=\lfloor Nx\rfloor+a_\varepsilon\left(\lfloor Nx\rfloor\right)-1
\ee
By Lemma \ref{lemma_ergodic} and assumption \eqref{assumption_uniform}, 
the r.h.s. of \eqref{for98} vanishes a.s. as $N\to+\infty$. 
Therefore, to establish \eqref{super_current_2}, it is enough to prove that
\be\label{super_current_3}
\liminf_{\varepsilon\to 0}\liminf_{\delta\to 0}\liminf_{N\to\infty}\Exp\left\{
N^{-1}\Gamma^\alpha_{z_N}\left(Nt,\eta_0^{\alpha,\delta}\right)
\right\}
\geq t(p-q)c
\ee
In order to prove the above, we write  (cf. \eqref{standard_current}) 
\begin{eqnarray}
&&\Exp\left\{
N^{-1}\Gamma^\alpha_{z_N}\left(Nt,\eta_0^{\alpha,\delta}\right)
\right\} \nonumber\\
& = &  N^{-1}\int_0^{Nt}\Exp \left\{
p\alpha(z_N)g\left[\eta_s^{\alpha,\delta}(z_N)\right]\right.\nonumber\\
& & \phantom{N^{-1}\int_0^{Nt}\Exp \left\{\right.}
\left.-q\alpha(1+z_N)g\left[(\eta_s^{\alpha,\delta}(1+z_N)\right]
\right\}ds\label{forlower}
\end{eqnarray}
Since by attractiveness we have $\eta_s^{\alpha,\delta}\geq\xi_s^{\alpha,\rho_c-\delta}$, 
in the above integral, we have the lower bound  (cf. \eqref{def_flux}) 
\begin{eqnarray}
\Exp \alpha(z_N)g\left[\eta_s^{\alpha,\delta}(z_N)\right] & \geq & 
\Exp \alpha(z_N)g\left[\xi_s^{\alpha,\rho_c-\delta}(z_N)\right]\nonumber\\
& = & \int_{\mathbf{X}}\alpha(z_N)g(\xi)d\mu^{\alpha,\rho_c-\delta}(\xi)=\overline{R}^{\,\,-1}(\rho_c-\delta)\label{lowerbound_integral}
\end{eqnarray}
On the other hand, by definitions \eqref{def_y_N} of $z_N$, \eqref{def:ell} 
of  $a_\varepsilon(\alpha)$,  and the inequality $g\leq 1$,  we have the upper bound
\be\label{upperbound_integral}
\Exp\left\{\alpha(1+z_N)g\left[(\eta_s^{\alpha,\delta}(1+z_N)\right]\right\}\leq c+\varepsilon
\ee
The above bounds \eqref{lowerbound_integral}--\eqref{upperbound_integral} imply 
that from \eqref{forlower} we get 
\[
\Exp\left\{
N^{-1}\Gamma^\alpha_{z_N}\left(Nt,\eta_0^{\alpha,\delta}\right)
\right\}\geq t[p\,\overline{R}^{\,\,-1}(\rho_c-\delta)-q(c+\varepsilon)]
\]
The limit \eqref{super_current_3} follows, since  $\overline{R}^{\,\,-1}(\rho_c)=c$
 (cf. \eqref{bysetting}). 
\end{proof} 
\mbox{}\\ \\
\appendix
%
\section{Proof of Lemma 2.1}
\label{app:justify}
Consider for instance the left-hand side $M_n$ of \eqref{justify_mean_beta}. By \eqref{mean_density_quenched},
$$
\Exp_{\mu_\beta^\alpha}(M_n)=\frac{1}{n+1}\sum_{x=0}^n R\left[\frac{\beta}{\alpha(x)}\right]\stackrel{n\to+\infty}{\longrightarrow}\overline{R}(\beta)
$$
On the other hand, 
$$
V(\beta):=\sum_{n=0}^{+\infty}n^2\theta_\beta(n)-R(\beta)^2
$$
is continuous on $[0,1)$, and, using Assumption \ref{assumption_ergo}, we have
$$
\mathbb{V}_{\mu_\beta^\alpha}(M_n)=\frac{1}{(n+1)^2}\sum_{x=0}^n V\left[\frac{\beta}{\alpha(x)}\right]\stackrel{n\to+\infty}{\sim}\frac{\overline{V}(\beta)}{n}
$$
where (with the conventions  \eqref{extension_theta}--\eqref{convention_2}) 
$$
\overline{V}(\beta):=\int_{[0,1]}V\left(\frac{\beta}{a}\right)dQ_0(a)=
\int_{[C,1]}V\left(\frac{\beta}{a}\right)dQ_0(a)\in[0,+\infty),\,\,\forall \beta\in[0,c)
$$
The conclusion then follows from Tchebychev's inequality.
%
\section{Proof of Lemma 3.2}
\label{app:uniform}
{\em Proof of (i).}   If $c<C$, by considering test functions $f$ 
supported either around $c$ or around $[C,1]$, we see that 
\eqref{eq:assumption_ergo} holds if and only if the contributions 
of each term on the r.h.s. of \eqref{inversion} to the empirical 
measures converges separately.
Following \eqref{sequence_alpha}, the contribution of the second term 
can only be a pointmass at $c$ and this requires $x_n/n$ to have 
a finite limit. If this pointmass were positive, by 
\eqref{sample_uniform}--\eqref{inversion}, the contribution of the first 
term would also contain a pointmass at $c$ in compensation, because 
$Q_0$ does not have such a pointmass. This is impossible because 
$F_{Q_0}^{-1}(u)$ always lies in the support of $Q_0$, hence in $[C,1]$. 
Conversely,  \eqref{growth_inc_x} implies that the second term of 
\eqref{inversion} does not contribute to the limits in 
\eqref{eq:assumption_ergo}. Therefore, \eqref{eq:assumption_ergo}
is equivalent to  \eqref{growth_inc_x} {\it plus} \eqref{empirical_uniform}.\\ \\
Assume all conditions hold. 
We  verify \eqref{empirical_uniform}. Condition \eqref{growth_inc} implies that
\be\label{subseq_ergo}
\lim_{n\to+\infty}\frac{1}{y_n}\sum_{x=0}^{y_n-1}\delta_{u(x)}=
\lim_{n\to-\infty}\frac{1}{y_n}\sum_{x=y_{n}+1}^{0}\delta_{u(x)}=\mathcal U(0,1)
\ee
Indeed, let $Q^n:=(y_n-y_0)^{-1}\sum_{x=y_0}^{y_n-1}\delta_{u(x)}$. 
For a nondecreasing function $f$ on $[0,1]$, 
\be\label{rectangle}
\int_{[0,1]}f(u)dQ^n(u)=\frac{\sum_{k=0}^{n-1}l_k i(l_k)}{\sum_{k=0}^{n-1}l_k}
\ee
where $l_k:=y_{k+1}-y_k$, and
\[
i(l)=\frac{1}{l}\sum_{i=0}^{l-1}f\left(\frac{i}{l}\right)
\]
is a rectangle approximation of $\int_0^1 f(u)du$ with an error  
bounded  by $||f||_\infty/l$. It follows that  \eqref{rectangle} 
approximates $\int_0^1 f(u)du$ with an error bounded by 
\[
||f||_\infty\frac{\sum_{k=0}^{n-1}l_k \frac{1}{l_k}}{\sum_{k=0}^{n-1}l_k}=||f||_\infty\frac{n}{y_n}
\]
which implies \eqref{subseq_ergo} by \eqref{growth_inc}. 
Condition \eqref{assumption_afgl_1} allows to fill the gap 
between \eqref{subseq_ergo} and \eqref{eq:assumption_ergo}. 
Indeed, for the above test function $f$,  
\[
\frac{1}{y_{p+1}}\sum_{k=0}^{y_p}f[u(k)]\leq\frac{1}{n}\sum_{k=0}^{n-1}f[u(k)]
\leq\frac{1}{y_{p}}\sum_{k=0}^{y_{p+1}}f[u(k)]
\]    
where $p$ is such that $y_p\leq n-1<y_{p+1}$, and \eqref{assumption_afgl_1} 
makes the ratio of the extreme terms tend to $1$ as $p\to\pm\infty$. Considering $f(u)=u$, we have
\[
\int_{[0,1]}f(u)dQ^n(u)=\frac{1}{y_n}\sum_{k=0}^{n-1}\sum_{i=0}^{l _k}\frac{i}{l_k}=\frac{1}{y_n}\sum_{k=0}^{n-1}\frac{l_k+1}{2}
=\frac{y_n+n}{2y_n}
\]
that converges to $\int_0^1udu$ only if \eqref{growth_inc} holds. 
If \eqref{assumption_afgl_1} fails, consider a subsequence of 
values of $y_{n+1}/y_n$ converging to $a\in(1,+\infty]$, then, 
for $f(u)=u$, and $z_n=(y_n+y_{n+1})/2$, we have
\be\label{empirical_z}
\frac{1}{z_n}\sum_{x=y_0}^{z_n}f[u(x)]=\frac{2}{y_n+y_{n+1}}\left(
\frac{y_n+n}{2}+s_n
\right)
\ee
where $s_n\sim l_n/8$ is the contribution of the sum 
between $x=y_n+1$ and $x=z_n$.  Hence, using \eqref{growth_inc}, 
the l.h.s. of \eqref{empirical_z} converges to 
\[
\frac{2}{1+a}\left[
\frac{1}{2}+\frac{a-1}{8}
\right]
=\frac{a+3}{4(a+1)}<1/2
\]
{\em Proof of (ii).} If $c<C$, this is a tautology.
If $c=C$, for $n\in\Z$, by \eqref{sample_uniform}--\eqref{inversion}, 
the sequence $(y_n)_{n\in\Z}$ satisfies $\alpha(y_n)=C=c$ 
and thus satisfies Assumption \ref{assumption_dense} by \eqref{assumption_afgl_1}.
Conversely, let $(t_n)_{n\in\Z}$ be such that $t_{n+1}/t_n\to 1$ 
and $\alpha(t_n)\to C$. The latter limit implies that $(t_n)_{n\in\Z}$
is extracted from a sequence of the form 
\be\label{set_z}z_n=y_n+\varepsilon_n(y_{n+1}-y_n)\ee
where $\varepsilon_n\to 0$ as $n\to\pm\infty$. If $y_{n+1}/y_n$ 
has a subsequence tending to $a\in[1,+\infty]$, 
the corresponding subsequence of $t_{n+1}/t_n$ converges to $a$. Thus $a=1$.\\ \\
\noindent{\bf Acknowledgments.}
We thank Gunter Sch\"utz for pointing us reference \cite{hs},
and for many interesting discussions.
This work was partially supported by laboratoire MAP5,
grants ANR-15-CE40-0020-02 and ANR-14-CE25-0011,
LabEx CARMIN (ANR-10-LABX-59-01), 
Simons Foundation Collaboration grant 281207 awarded to K. Ravishankar.
We thank
Universit\'{e}s Clermont Auvergne and Paris Descartes for hospitality.
This work was partially carried out during C.B.'s 2017-2018
d\'el\'egation CNRS, whose support is acknowledged.
Part of it was done during the authors' stay at the
Institut Henri Poincar\'e (UMS 5208 CNRS-Sorbonne Universit\'e) - Centre Emile Borel for
the trimester ``Stochastic Dynamics Out of Equilibrium'', and during the authors' stay at NYU Shanghai. 
The authors thank these institutions for hospitality and support. 
\end{document}